\documentclass[12pt]{amsart}
\usepackage{amscd}
\usepackage{amssymb,amsmath,amsthm}
\usepackage[all]{xy}

\theoremstyle{plain}
\newtheorem{thm}{Theorem}[section]
\newtheorem{lem}[thm]{Lemma}

\newtheorem{cor}[thm]{Corollary}
\newtheorem{defn-lem}[thm]{Definition-Lemma}

\newtheorem{prop}[thm]{Proposition}

\theoremstyle{definition}
\newtheorem{defn}[thm]{Definition}

\newtheorem{rem}[thm]{Remark}

\def\md #1#2#3#4#5 {\left(
                        \begin{matrix}
             #1 & #2 \\
             #3 & #4
                        \end{matrix}
                      \right)- #5}

\def\ma #1#2#3#4 {\left(
                        \begin{matrix}
             #1 & #2 \\
             #3 & #4
                        \end{matrix}
                      \right)}
\def \mu  {\mathcal{M}(C(X,B))}

\def\Index{\operatorname{Index}}

\def\Ad{\operatorname{Ad}}
\def\id{\operatorname{id}}

\newcommand{\mc}{\mathcal}
\begin{document}
\title [A tracially sequentially-split $\sp*$-homomorphisms between $C\sp*$-algebras II]
       { On Dualities of actions and inclusions}

\begin{abstract}
 Following the results known in the case of a finite abelian group action on $C\sp*$-algebras we prove the following two theorems;
\begin{itemize}
\item an inclusion $P\subset A$ of (Watatani) index-finite type has the Rokhlin property (is approximately representable) if and only if the dual inclusion is approximately representable (has the Rokhlin property).
\item an inclusion $P\subset A$ of  (Watatani) index-finite type has the tracial Rokhlin property (is tracially approximately representable) if and only if the dual inclusion is tracially approximately representable (has the tracial Rokhlin property).
\end{itemize}
Moreover, we provide an alternate proof of Phillips' theorem about the relations between tracial Rokhlin action and tracially approximate representable dual action using a new conceptual framework suggested by authors.
\end{abstract}

\author { Hyun Ho \, Lee and Hiroyuki Osaka}

\address {Department of Mathematics\\
          University of Ulsan\\
         Ulsan, South Korea 44610 \\
 and \\
School of Mathematics\\ Korea Institute of Advanced Study\\
Seoul, South Korea, 130-772 }
\email{hadamard@ulsan.ac.kr}

\address{Department of Mathematical Sciences\\
Ritsumeikan University\\
Kusatsu, Shiga 525-8577, Japan}
\email{osaka@se.ritsumei.ac.jp}


\date{}
\thanks{The first author's research was supported by Basic Science Research Program through the National Research Foundation of Korea(NRF) funded by the Ministry of Education(NRF-2015R1D1A1A01057489)\\
The second author's research was partially supported by the JSPS grant for Scientific Research No. 17K05285.}
\maketitle

\section{Introduction}
In \cite{Izumi:finite} M.Izumi introduced the Rokhlin property for a finite group action on $C\sp*$-algebras. In addition, he showed that a finite abelian group action has the Rokhlin property if and only if the dual action is approximately representable which means that the action is strongly approximately inner. This observation is sometimes useful and easier to verify rather than the Rokhlin property itself. Based on Izumi's works, the second named author, Kodaka, and Teruya extend notions of the Rokhlin property and the approximate representability for inclusions of unital $C\sp*$-algebras.  In this note we show that an inclusion has the Rokhlin property if and only if its dual inclusion is approximately representable. 

On the other hand, a great success of classifying nuclear simple tracially AF $C\sp*$-algebras satisfying the UCT \cite{Lin:tracial, Lin:classification} provided  a conceptual revolution in the Eliiott program; no inductive limit structure is assumed. The added flexibility to allow ``small'' tracial error in the local approximation  led to important advances   in the theory of $C\sp*$-algebras. It is thus desirable to expect  ``tracial'' versions of other $C\sp*$-algebra concepts, see \cite{HO},\cite{EN} for instance. But also in part due to the fact that often the Rokhlin property imposes several restrictions  on the K-theory of the original algebra and the K-theory of the crossed product algebra it is expected to have a less restrictive or tracial version of the Rokhlin property, and such a notion was suggested by N.C. Phillips, which is called the tracial Rokhlin property. He then extended many statements appeared in \cite{Izumi:finite} when a finite group action on the infinite dimensional simple separable unital $C\sp*$-algebra has the tracial Rokhlin property. In particular, he shows that a finite abelian group action has the tracial Rokhlin property if and only if the dual group action is tracially approximately representable (see Theorem \ref{T:dualityofgroupaction} or \cite[Theorem 3.11]{Phillips:tracial}). In this note we are going to define the tracial versions of the Rokhlin property and the approximate representability for inclusions of unital $C\sp*$-algebras and show a duality between them. 

One might think that it is not difficult to find analogous notions while extending original ones. But in our case, though the pair of the crossed product algebra and the original algebra is a standard model for inclusions of $C\sp*$-algebras, it is not obvious which part of the definition is made flexible; the tracial Rokhlin property is straightforward but not at all for the tracial approximate representability. Thus we wish to point out a philosophical guiding principle; a regularity property involving two objects can be expressed as a certain property of a map between two objects. More precisely, Barlak and Szab\'{o} \cite{BS} provide an unified conceptual framework to deal with the permanence of various regularity properties from the target algebra to the domain algebra in the name of  sequentially split $*$-homomorphisms between $C\sp*$-algebras. This concept is actually originated from Toms and Winter's characterization of $D$-stability, where $D$ is a unital strongly self-absorbing $C\sp*$-algebra, but nicely fits into the Rokhlin property as well. A tracial version of the concept has been suggested by the authors in \cite{LeeOsaka} and  turned out to work nicely with tracial versions of regularity properties, for instance tracial $\mc{Z}$-stability and tracial Rokhlin property. In this note we solidify our guiding principle by exhibiting that the Rokhlin property and the approximate representability of an inclusion of unital $C\sp*$-algebras $P\subset A$ could be characterized by the existence of a map from $A$ to the sequence algebra of $P$ and a map from the $C\sp*$-basic construction to the sequence algebra of $A$ (see Proposition \ref{P:Rokhlinpropertyviamap} and Proposition \ref{P:approximaterepresentable}). Accordingly, we characterize the tracial version of the Rokhlin property and the approximate representability of the action of a finite abelian group $G$ on  a unital separable  $C\sp*$-algebra $A$ by the existence of a map from $C(G)$ the algebra of continuous functions on $G$ to the central sequence algebra of $A$ and a map from the crossed product algebra to the sequence algebra of $A$ as is expected.  Then we provide an alternate proof of the duality result of Phillips using these characterizations.  Moreover, we show an interplay between group actions and inclusions of $C\sp*$-algebras based on our duality results for both the strict case and the tracial case.

\section{Tracially sequentially-split homomorphism between $C\sp*$-algebras}
In this section we briefly review the definition of tracially sequentially-split map between separable $C\sp*$-algebras from \cite{LeeOsaka} and introduce notations which will be used throughout the note.
 
For a $C\sp*$-algebra $A$, we set the $C\sp*$-algebra of bounded sequence over $\mathbb{N}$ with values in $A$ and the ideal of sequences converging to zero as follows;
\[l^{\infty}(\mathbb{N}, A)=\{ (a_n)\mid  \{\|a_n\|\} \,\text{bounded} \}\]
\[c_0(\mathbb{N}, A)=\{(a_n)\mid \lim_{n\to \infty}\|a_n \|=0 \}. \]
Then we denote by $A_{\infty}=l^{\infty}(\mathbb{N}, A)/c_0(\mathbb{N}, A)$ the sequence algebra of $A$ with the norm  of $a$ given by $\limsup_n \|a_n\| $, where $(a_n)_n$ is a representing sequence of $a$. We can embed $A$ into $A_{\infty}$ as a constant sequence, and we denote the central sequence algebra of $A$ by 
\[A_{\infty} \cap A'.\]
For an automorphism of $\alpha$ on $A$, we also denote by $\alpha_{\infty}$ the induced automorphism on $A_{\infty}$ and $A_{\infty}\cap A'$ without confusion. 

We save the notation $\lesssim$ for the Cuntz subequivalence of two positive elements; for two positive elements $a, b$ in A  $a \lesssim b$ if there is a sequence $(x_n)$ in $A$ such that $\| x_nbx_n^*-a\| \to 0$ as $n\to \infty$. Often when $p$ is a projection, we see that $p \lesssim a$ if and only if  there is a projection in the hereditary $C\sp*$-subalgebra generated by $a$  which is Murray-von Neumann equivalent to $p$. For more details, we refer \cite{Cu, Ro:UHF1, Ro:UHF2} for example.

\begin{defn}
Let $A$ and $B$ be (unital) separable $C\sp*$-algebras. A $*$-homomorphism $\phi:A \to B$ is called tracially sequentially-split, if for every positive nonzero element $z \in A_{\infty}$  there exist a $*$-homomorphism $\psi: B \to A_{\infty}$ and a nonzero projection $g\in A_{\infty}\cap A'$ such that 
\begin{enumerate}
\item $\psi(\phi(a))=ag$ for each $a\in A$, 
\item $1_{A_{\infty}} -g$ is Murray-von Neumann equivalent to a projection in a hereditary $C\sp*$-subalgebra $\overline{zA_{\infty}z}$ in $A_{\infty}$.
\end{enumerate}
We also consider the following alternative stronger condition to replace (2) in the above definition;
for any $\delta >0$ there exist a $*$-homomorphism $\psi:B \to A_{\infty}$ and  a projection $g \in A_{\infty}\cap A'$ such that  
\begin{enumerate}
\item[(1)] $\psi(\phi(a))=ag$ for each $a\in A$,
\item[(2)']   $\tau(1-g)< \delta$ for all $\tau \in T(A_{\infty})$ the tracial states of $A_{\infty}$.
\end{enumerate}
\end{defn}

Since $(\psi\circ \phi)(a)=a-a(1_{A_{\infty}}-g)$, we can view $\psi\circ\phi$  equal to $\iota$ up to ``tracially small error''. If $A$ and $B$ are unital $C\sp*$-algebras and $\phi$ is unit preserving, then $g=\psi(1_B)$. Moreover, if $g=1_{A_{\infty}}$, then $\phi$ is called a (strictly) sequentially split $*$-homomorphism following Barlak and Szabo \cite{BS} since the second condition is automatic . The $\psi$ in the above definition is called a tracial approximate left-inverse of $\phi$. Although the diagram below is not commutative, we still use it to symbolize the definition of a tracially sequentially-split map $\phi$ and its tracial approximate left inverse $\psi$;

\begin{equation}\label{D:diagram}
\xymatrix{ A \ar[rd]_{\phi} \ar[rr]^{\iota} && A_{\infty} &\\
                          & B \ar[ur]_{\psi} \ar@{}[u]|{\text{tracially}\circlearrowleft} }
\end{equation}

\begin{defn}
A $C\sp*$-algebra $A$ has the property (SP) if any nonzero hereditary $C\sp*$-subalgebra of $A$ has a nonzero projection. 
\end{defn}
\begin{prop}
Let $\phi:A\to B$ be a tracially sequentially-split $*$-homomorphism. Then $A$ has the property (SP) or $\phi$ is a (strictly) sequentially split $*$-homomorphism.
\end{prop}
\begin{proof}
Suppose that $A$ has no property (SP), then $A_{\infty}$ has no property (SP). Then there is a positive nonzero element $x$ in $A_{\infty}$ which generates a hereditary subalgebra that contains no nonzero projections.
Then since $\phi:A \to B$ is tracially sequentially-split, there are a projection $g\in A_{\infty} \cap A'$ and a $*$-homomorphism $\psi:B\to A_{\infty}$ such that 
\begin{enumerate}
\item $\psi(\phi(a))=ag$ for all $a\in A$,
\item $1-g $ is Murray-von Neumann equivalent to a projection in $\overline{xA_{\infty}x}$.
\end{enumerate} 
The second condition implies $1-g=0$ so that $\psi(\phi(a))=a$, thus $\phi$ is strictly sequentially split. 
\end{proof}

The following proposition which is a tracial version of \cite[Lemma 2.4]{BS} will be used in Section \ref{S:tracial}. 
\begin{prop}\label{P:firstfactorembedding}
Let $C$ and $A$ be unital $C\sp*$-algebras.  Suppose that for any nonzero positive element $z\in A_{\infty}$ there exists a map $\phi: C \to A_{\infty}\cap A'$ such that 
 $1-\phi(1_C)$ is Murray-von Neumann equivalent to a projection in $\overline{zA_{\infty}z}$.
Then the first factor embedding $\id_A\otimes 1_C:A \to A\otimes C$ is tracially sequentially-split. Moreover, the converse is also true.
\end{prop}
\begin{proof}
Let us denote by $m$ the map from $A \otimes (A_{\infty}\otimes A')$ to $A_{\infty}$ which is defined by 
\[ a\otimes [(a_n)] \to [(aa_n)].\] 
For any nonzero positive element $x$ in $A_{\infty}$  we take a map $\phi:C \to A_{\infty}$ as above and define $\psi$ as the composition of two maps $m$ and $\id_A \otimes \phi$, i.e., $\psi=m\circ (\id_A \otimes \phi)$. It follows immediately that $\psi(1_A \otimes 1_C)=\phi(1_C)$.
Then it is easily checked that
\begin{enumerate}
\item $(\psi \circ (\id_A\otimes 1_C)(a)=\psi(a\otimes 1_C)=a\phi(1_C)=a\psi(1_A\otimes 1_C)$ for any $a\in A$,
\item $1-\psi(1_A\otimes 1_C)=1-\phi(1_C)$ is Murray-von Neumann equivalent to a projection in $
\overline{xA_{\infty}x}$.
\end{enumerate}

Conversely,  for any positive nonzero element $x$ in $A_{\infty}$ we consider a tracial approximate inverse $\psi$ for $\id_A\otimes 1_C$. \\
Let $\phi(c)=\psi(1_A\otimes c)$. Obviously, $\phi$ is a $*$-homomorphism from $C$ to $A_{\infty}$. Moreover,  
\[\begin{split}
\phi(c)a&=\psi(1_A\otimes c)a=\psi(1_A\otimes c)\psi(1_A \otimes 1_C)a\\
&=\psi(1_A\otimes c )\psi(a\otimes 1_C)\\
&=\psi(a\otimes 1_C)\psi(1_A \otimes c)\\
&=a\psi(1_A\otimes 1_C)\psi(1_A\otimes c)\\
&=a\phi(c)
\end{split}\]
Therefore $\phi(C) \subset A_{\infty}\cap A'$. Finally, $1-\phi(1_C)=1-\psi(1_A\otimes 1_C)$ is Murray-von Neumann equivalent to a projection in $\overline{xA_{\infty}}x$. 
\end{proof}

When $A$ and $B$ are equipped with group actions,  instead of ordinary $*$-homomorphism we consider equivariant ones to define the equivariant version of a tracially sequentially-split map.

\begin{defn}
Let A and B be separable unital $C\sp*$-algebras and $G$ a finite group. Let $\alpha:G \curvearrowright 􏰏A $ and $\beta:G \curvearrowright 􏰏B$ be two actions. An equivariant $*$-homomorphism $\phi : (A,\alpha) \to (B,\beta)$ is called  $G$-tracially sequentially-split if for every nonzero positive element $x$ in $A_{\infty}$ there exist an equivariant $*$-homomorphism $\psi:(B,\beta) \to (A_{\infty}, \alpha_{\infty})$ and a projection $g$ in $A_{\infty}\cap A'$ such that 
\begin{enumerate}
\item $\psi (\phi (a))=ga=ag$,  
\item $1-g$ is Murray-von Neumann equivalent to a projection in $\overline{xA_{\infty}x}$.
\end{enumerate} 
\end{defn}
 We shall use the following diagram to describe the equivariant case of tracially sequentially-split map $\phi$ and its tracial approximate left inverse $\psi$; 
\begin{equation}\label{D:diagram}
\xymatrix{ (A, \alpha) \ar[rd]_{\phi} \ar@{-->}[rr]^{\iota} && (A_{\infty}, \alpha_{\infty}) &\\
                          & (B, \beta) \ar[ur]_{\psi} \ar@{}[u]|{\text{tracially}\circlearrowleft}} 
\end{equation}
The following is a straightforward generalization of Proposition \ref{P:firstfactorembedding} to the equivariant case. We remark that a group $G$ could be compact for this statement but we restrict ourselves to finite groups with applications in mind.  
\begin{prop}\label{P:firstfactorembeddingequivariant}
Let $C$ and $A$ be unital $C\sp*$-algebras.  Let $\alpha:G \curvearrowright A$ and $\beta:G \curvearrowright C$ be two action of a finite group $G$. Suppose that for any nonzero positive element $x\in A_{\infty}$ there exists an equivariant $*$-homomorphism $\phi: (C,\beta) \to (A_{\infty}\cap A', \alpha_{\infty})$ such that $1-\phi(1_C)$ is Murray-von Neumann equivalent to a projection in $\overline{zA_{\infty}z}$. Then the first factor embedding $\id_A\otimes 1_C:(A, \alpha) \to (A\otimes C, \alpha\otimes \beta)$ is tracially sequentially-split. Moreover, the converse is also true.
\end{prop}
\begin{proof}
The proof is almost same, the only thing to be careful is that the map $m: (A_{\infty}\cap A')\otimes A \to A_{\infty}$ in Proposition \ref{P:firstfactorembedding} is equivariant with respect to actions and this is easily checked. 
\end{proof} 
\section{Dualities of actions and inclusions: The strict case}

\begin{defn}[Watatani\cite{W:index}]
Let $P\subset A$ be an inclusion of unital $C\sp*$-algebras and $E:A \to P$ a conditional expectation. Then we way that $E$ has a quasi-basis if there exist elements $\{(u_k,v_k)\}$ for $k=1,\dots, n$ such that for any $x\in A$
\[ x=\sum_{j=1}^nu_jE(v_jx)=\sum_{j=1}^n E(xu_j)v_j.\]
In this case, we define the Watatani index  of $E$ as 
\[\Index E= \sum_{j=1}^n u_jv_j.\] In other words, we say that $E$ has a finite index if there exist a quasi-basis. 
\end{defn}
It is proved that once we know the existence of a quasi-basis  then a quasi-basis can be chosen as $\{(u_1, u_1^*), \dots, (u_n, u_n^*)\}$ so that $\Index E$ is a nonzero positive element in $A$ commuting with $A$.  Thus if $A$ is simple,  it is a nonzero positive scalar.\\
We also recall Watatani's notion of the $C\sp*$-basic construction of the above triple $(P, A, E:A \to P)$; since we only consider the case that the conditional expectation $E:A  \to P$ is of  index-finite type , we do not distinguish the reduced construction and maximal one.    
\begin{defn}
Let $P\subset A$ be an inclusion of unital $C\sp*$-algebras and $E:A \to P$ a conditional expectation.
Now we assume $E$ is faithful. Let $\mc{E}_{E}$ be the Hilbert $P$-module completion of $A$ by the norm given by a $P$-valued Hermitian bilinear form $\langle x, y \rangle_P=E(x^*y)$ for $x,y \in A$. As usual $\mc{L}(\mc{E}_E)$ will be the algebra of  adjointable bounded operators on $\mc{E}_E$. There are an injective $*$-homomorphism $\lambda:A \to \mc{L}(\mc{E}_E)$ defined by a left multiplication and the natural inclusion map $\eta_E$ from $A$ to $\mc{E}_E$. 
The the \emph{Jones projection} $e_P$ is defined by 
\[ e_p(\eta_E(x))=\eta_E(E(x)).\]
Then the $C\sp*$-basic construction is the $C\sp*$-algebra given by 
\[C^* \langle A, e_P \rangle =\{\sum_{i=1}^n \lambda(x_i)e_P \lambda(y_i)\mid x_i, y_i \in A, n\in \mathbb{N} \}. \]
When $E$ is of index-finite type, there is a dual conditional expectation $\widehat{E}$ from $C^* \langle A, e_P \rangle $ onto $A$ such that for $x,y \in A$
\[ \widehat{E}(\lambda(x)e_p\lambda(y))=(\Index E)^{-1} xy. \] Moreover, $\widehat{E}$ is also of index-finite type and faithful.
\end{defn}
 From now on, otherwise stated, we only consider a faithful conditional expectation. 
   
\begin{defn}[Osaka, Kodaka, and Teruya]
A conditional expectation $E:A \to P$ of  index-finite type has the Rokhlin property if there is a projection $e\in A_{\infty}\cap A'$ such that $E_{\infty}(e)=(\Index E)^{-1}$ and the map $A \ni x \mapsto xe$ is injective. We call  $e$ a Rokhlin projection.  
\end{defn} 
\begin{defn}(Osaka and Teruya)\label{D:approximaterepresentability}
A conditional expectation $E:A \to P$  of  index-finite type is approximately representable if there exists a projection $e \in P_{\infty}\cap P'$ and a finite set of elements $\{u_i\}\subset A$ such that 
\begin{enumerate}
\item $exe=E(x)e$ for every $x\in A$, 
\item $\sum_iu_ieu_i^*=1$
\item the map $x \mapsto xe$ is injective for $x\in P$.  
\end{enumerate}
\end{defn}
\begin{rem}
The condition (2) in the above didn't appear in the original definition of  approximate representability  \cite{OKT:Rokhlin},  but we include it for the definition.   
\end{rem}

The following proposition gives us a characterization of the approximate representability for an inclusion of $C\sp*$-algebras reflecting the whole theme of our note. 
\begin{prop}\label{P:approximaterepresentable}
A conditional expectation $E:A \to P$ of index-finite type is approximately representable  if and only if there is a unital injective $*$-homomorphism $\psi:C^*\langle A, e_p \rangle \to A_{\infty}$ such that $\psi(e_p) \in P_{\infty}\cap P'$, $\psi(x)=x$ for every $x\in A$ 
\end{prop}
\begin{proof}
Suppose that $E:A\to P$ is approximately representable. We note that condition (2) implies that the $C\sp*$-algebra generated by $\{xey\mid x, y \in A\}$ is nondegenerate. It follows that $(\iota, e)$ is a nondegenerate covariant representation where $\iota$ is the natural embedding of $A$ into $A_{\infty}$ by a constant sequence.  By \cite[Proposition 2.2.11]{W:index},  we have a map $\psi$ from the $C\sp*$-basis construction $C^*\langle A, e_p \rangle$ to $\overline{AeA}$ such that $\psi(xe_py)=xey$ for all $x, y \in A$.  Moreover the condition (3) implies that $\psi$ is injective. Note that 
\[\begin{split}
ae=&\sum_i u_ieu_i^*ae\\
=&\sum_i u_i E(u_i^*a)e
\end{split}\]
Then we have $\psi(ae_p)=\psi(\sum_i u_iE(u_i^*a)e_p)$. By the injectivity of $\psi$, $ae_p=\sum_i u_iE(u_i^*a)e_p$. It follows that $a=\sum_i u_iE(u_i^*a)$. Similarly, we have $a=\sum_i E(au_i)u_i^*$.
Thus $\{(u_i, u_i^*)\}$ is a quasi-basis, and $\psi$ is a unital map. 

 Conversely, if we have a unital injective map $\psi:C\langle A, e_p\rangle \to A_{\infty}$ such that $\psi(e_p) \in P_{\infty}\cap P$ and $\psi(a)=a$ for all $a\in A$, then from $e_pae_p=E(a)e_p$ for all $a\in A$ we also have $eae=E(a)e$ for all $a\in A$.  In addition,  $pe=0$ implies that $\psi(pe_p)=0$. Thus $pe_p=0$ by the injectivity of $\psi$. 
If $\{(u_i,u_i^*)\}$ is a quasi-basis for $E$, then $\sum_i u_i e_p u_i^*=1$ in $C^*\langle A, e_p \rangle$. 
Thus\[  \psi(\sum_i u_i e_p u_i^*)=\sum_i u_i e u_i^*=1.\] 
\end{proof}
\begin{rem}
If a conditional expectation $E:A \to P$ of index-finite type is approximately representable, and $\{u_i\} \subset A$ is the set satisfying $\sum_i u_ie u_i^*=1$, then $\{(u_i,u_i^*)\}$ is a quasi-basis for $E$. 
\end{rem}
\begin{cor}
If $\eta:A \to C^*\langle A, e_p \rangle$ the natural embedding for the inclusion $P\subset A$ of index-finite type, and the conditional expectation $E:A \to P$ is approximately representable, then $\eta$ is a sequentially split $*$-homomorphism. 
\end{cor}

Similarly we characterize an inclusion $P\subset A$ with the Rokhlin property using a map from $A$ to $P_{\infty}$ in the same spirit of Proposition \ref{P:approximaterepresentable}. 
\begin{prop}\label{P:Rokhlinpropertyviamap}
Let $P\subset A$ be an inclusion of unital $C\sp*$-algebras and $E:A\to P$ be a conditional expectation of  index-finite type.  Suppose further $A$ is simple. Then $E$ has the Rokhlin property if and only if  there exist  a  map $\beta:A \to P_{\infty}$ and a projection $e\in A_{\infty}\cap A'$ such that
\begin{enumerate}
\item $ae=\beta(a)e$ for all $a\in A$,
\item $(\Index E)ee_pe=e$,
\item $ye=ze$ implies that $y=z$ for all $y,z\in P_{\infty}$,
\end{enumerate} 
 \end{prop}
\begin{proof}
For ``only if'' part, see \cite[ Theorem 5.7]{OKT:Rokhlin}. In fact, take $e$ a Rokhlin projection and it follows that $(2)$ is true with the fact that $(\Index E)E_{\infty}(e)=1$, and for $x\in A_{\infty}$
\[\begin{split}
xe&=(\Index E)\widehat{E}_{\infty}(e_pxe)\\
   &=(\Index E)^2\widehat{E}_{\infty}(e_p x e e_pe)\\
   &=(\Index E)^2 \widehat{E}_{\infty}(E_{\infty}(xe)e_p e)\\
  &=(\Index E )E_{\infty}(xe)e. 
\end{split}
\]
Thus we define $\beta(x)=(\Index E) E_{\infty}(xe)$. Finally, 
so that if $y$ and $z$ satisfy $ye=ze$, then 
\[\begin{split}
y=&y(\Index E)E_{\infty}(e)=(\Index E)E_{\infty}(ye)\\
=&(\Index E)E_{\infty}(ze)\\
=&z(\Index E)E_{\infty}(e)=z
\end{split}\]  

Conversely, suppose that we have a map $\beta$ and a projection $e$ satisfying (1)-(3). Then (2) implies  that $ae=(\Index E) E_{\infty}(ae)e$ as shown above. Then the conditions (1), (3) imply that $\beta(a)=(\Index E) E_{\infty}(ae)$. Then (1) and (3) again imply that $\beta(1)=1$ so that $(\Index E)E_{\infty}(e)=1$. Thus $e$ is a Rokhlin projection and the inclusion $P \subset A$ has the Rokhlin property. 
\end{proof}

Whenever we have a finite group action $\alpha:G \curvearrowright A$, then we have an inclusion $A^{\alpha}  \subset A$ where $A^{\alpha}$ is the fixed point algebra with a faithful conditional expectation $E:A \to A^{\alpha}$ defined by $\displaystyle E(a)=\frac{1}{|G|} \sum_{g} \alpha_g(a)$.  It is not always true that $E$ is of  index-finite type. But the following is known.  
\begin{thm} \cite[Theorem 4.1]{JP:saturated} \label{T:saturated}
Let $\alpha:G \curvearrowright  A$ be an action of a finite group $G$ on $A$. Then 
$\alpha$ is saturated if and only of  $E:A \to A^{\alpha}$ defined by $\displaystyle E(a)=\frac{1}{|G|} \sum_{g} \alpha_g(a)$is of  index-finite type.
\end{thm}
\begin{defn}[M. Izumi]
Let $\alpha:G \curvearrowright  A$ be an action of a finite group $G$ on $A$. We say  $\alpha$ has the Rokhlin property if there is a partition of unity $\{e_g\}_{g\in G}$ of projections in $A_{\infty}\cap A'$ such that for all $g,h\in G$
\[\alpha_{\infty, h}(e_g)=e_{hg}\]
where $\alpha_{\infty}$ is the induced action. 
\end{defn}
If $\alpha$ has the Rokhlin property, then $\alpha$ is outer, thus saturated.  By Theorem \ref{T:saturated} $E:A \to A^{\alpha}$ defined by $\displaystyle E(a)=\frac{1}{|G|} \sum_{g} \alpha_g(a)$ is of finite index. 
Then the following is a characterization of  the Rokhlin property of an action in term of inclusion of $C\sp*$-algebras due to Osaka, Kodaka, and Teruya. 

\begin{prop} \cite[Proposition 3.2]{OKT:Rokhlin}\label{P:ERokhlin}
Let $G$ be a finite group,  $\alpha:G \curvearrowright  A$ be an action of a finite group $G$ on a simple unital $C\sp*$-algebra $A$, and $E$ the conditional expectation defined by $\displaystyle E(a)=\frac{1}{|G|} \sum_{g} \alpha_g(a)$. Then $\alpha$ has the Rokhlin property if and only if $E$ has the Rokhlin property.
\end{prop}

In the following we present a parallel result for the approximately representable action $\alpha:G \curvearrowright A$. In other words, we would like to characterize the approximate representability of $\alpha$ in terms of the conditional expectation $E$ from $A$ onto the fixed point algebra $A^{\alpha}$. 

\begin{defn}(M. Izumi)\label{D:Approximaterepresentability}
Let $\alpha:G \curvearrowright  A$ be an action of a finite group $G$ on $A$. When $G$ is abelian, then the action $\alpha$ is called approximately representable if there is a unitary representation $\omega:G \to (A^{\alpha})_{\infty}$ such that for all $a\in A$
\[\alpha(a)=\omega_g a \omega^*_g. \]
\end{defn}
A tracial version the following theorem will be proved later, but at this moment we need it as a tool for our goal.   
\begin{thm}\cite[Lemma 3.8]{Izumi:finite}\label{T:dualityoffiniteabeliangroup}
Let $\alpha:G \curvearrowright  A$ be an action of a finite abelian group $G$ on a unital $C\sp*$-algebra $A$ and $\widehat{\alpha} \curvearrowright A\rtimes_{\alpha} G$ the dual action of $\widehat{G}$ on the crossed product algebra.  
\begin{enumerate}
\item $\alpha$ has the Rokhlin property if and only if $\widehat{\alpha}$ is approximately representable.
\item $\widehat{\alpha}$ has the Rokhlin property if and only if $\alpha$ is approximately representable.
\end{enumerate}
\end{thm}
\begin{proof}
For the proof, we rather recommend \cite[Theorem 4.27]{BS} which covers even for a  second countable compact abelian group with its dual as a discrete countable abelian group. 
\end{proof}

Although we don't use the following observation later, we include it since it provides a clue to  what we want to obtain.
\begin{prop}
Let $G$ be a finite abelian group and $\alpha:G \curvearrowright A$ approximately representable. 
Then the conditional expectation $E:A \to A^{\alpha}$ defined by $\displaystyle E(a)=\frac{1}{|G|} \sum_{g} \alpha_g(a)$ satisfies a covariant relation (see (1) of Definition \ref{D:approximaterepresentability}). i.e. there is a projection $e \in (A^{\alpha})_{\infty}\cap (A^{\alpha})^{'}$ such that for all $a\in A$
\[eae=E(a)e.\]
\end{prop}
\begin{proof}
Let $e=\dfrac{1}{|G|}\sum_g \omega_g$. Then $e^*=\dfrac{1}{|G|}\sum_g \omega_{g^{-1}}=e$ and 
\[ \begin{split}e^2&=\dfrac{1}{|G|}\sum_g \omega_g \dfrac{1}{|G|}\sum_h \omega_h \\
&=\dfrac{1}{|G|^2}\sum_{g,h} \omega_{gh}\\
&= \dfrac{1}{|G|}\sum_g  \dfrac{\sum_h \omega_{gh}}{|G|}\\
&=\dfrac{1}{|G|}\sum_g e=e.
\end{split}\] Thus $e$ is a projection. Moreover, for $a\in A^{\alpha}$ 
\[ea=\dfrac{1}{|G|}\sum_g \omega_g a= \dfrac{1}{|G|}\sum_g \alpha_g(a)\omega_g=a \dfrac{1}{|G|}\sum_g \omega_g=ae.\] Hence $e \in (A^{\alpha})_{\infty}\cap (A^{\alpha})^{'}.$ 
Finally, for $x\in A$
\[\begin{split}
exe&= \left(\dfrac{1}{|G|}\sum_g \omega_g\right) \, x \, \left(\dfrac{1}{|G|}\sum_h \omega_h\right)\\
&=\frac{1}{|G|^2}\sum_{g,h}\alpha_g(x)\omega_{gh}\\
&=\frac{1}{|G|^2}\sum_{h,s}\alpha_{h^{-1}s}(x)\omega_{s}\\
&=\frac{1}{|G|^2}\sum_{s,h}\alpha_{sh^{-1}}(x)\omega_{s}\\
&=\frac{1}{|G|}\sum_s \alpha_s \left( \frac{\sum_{h}\alpha_{h^{-1}}(x)}{|G|}\right)\omega_s\\
&=\frac{1}{|G|}\sum_s E(x) \omega_s=E(x)e
\end{split}\]
\end{proof}

We need the following theorem as a final piece of our puzzle, which is of independent interest in the sense that it is an extension of Theorem \ref{T:dualityoffiniteabeliangroup}.  We remark that the following statement was claimed in \cite{OKT:Rokhlin} but the proof was incomplete in view of Definition \ref{D:approximaterepresentability}.

\begin{thm}\label{T:dualityofinclusion}
Let $P\subset A$ be an inclusion of unital $C\sp*$-algebras and $E$ a conditional expectation from A onto P with a finite index. Let $B=C^*\langle A, e_P\rangle$ be the basic construction and $\widehat{E}$ the dual conditional expectation of $E$ from $B$ onto $A$. Then
\begin{enumerate}
\item  $E$ has the Rokhlin property if and only if $\widehat{E}$ is approximately representable; 
\item  $E$ is approximately representable if and only if $\widehat{E}$ has the Rokhlin property.
\end{enumerate}
\end{thm}
\begin{proof}
(1): Let $e$ be the Rokhlin projection  in $A_{\infty}\cap A'$. We will show that there exists a finite set $\{u_i\} \subset B$ such that \[\sum_i u_i e u_i^*=1. \] Let $\{(v_i,v_i^*) \}$ a quasi-basis for $E$. Then put $u_i=\sqrt{\Index E} v_i e_p$.
Then \[
\begin{aligned}
\sum_i u_ieu_i^*&= \sum_i  \Index E (v_i e_P e e_P v_i^*) \\
&= \sum_i \Index E( v_i E_{\infty}(e)e_P v^*_i)\\
&=\sum_i v_i e_Pv_i^*=1
\end{aligned}
\]
Conversely, if $\widehat{E}$ is approximately representable, then we have a projection $e \in A_{\infty}\cap A$ and a finite set $\{u_i\} \in B$ such that 
\begin{equation}
eze=\widehat{E}(z)e \quad \forall z \in B,
\end{equation}
\begin{equation}
\sum_i u_ieu_i^*=1
\end{equation}
Write $u_i=\sum_j a_{ij} e_P b_{ij}$ for some $a_{ij}, b_{ij} \in A$. Let $w_i=\sum_j E(a_{ij})b_{ij}$. 

Note that \[\begin{split}
\sum_i E_{\infty}(w_iew_i^*)e_P&=\sum_i E_{\infty} \left( \left(\sum_j E(a_{ij})b_{ij} \right) e \left(\sum_k b_{ik}^* E(a_{ik}^*)\right) \right) e_P\\
&=\sum_i e_P \left(\left(\sum_j E(a_{ij})b_{ij} \right) e \left(\sum_k b_{ik}^* E(a_{ik}^*)\right)\right)e_P\\
&=\sum_i e_P \left(\sum_j a_{ij}e_Pb_{ij}\right) e \left(\sum_{k}b_{ik}^* e_P a_{ik}^*\right) e_P\\
&=e_P. 
\end{split}
\]
It follows that $\sum_i E(w_iew_i^*)=1$.  If $x\in P_{\infty}$ such that $xe=0$, then 
\[\begin{split}
0=xe=xe\sum_i w_iw_i^*
\end{split}
\]
Therefore, $x=xE_{\infty}(e\sum_i w_iw_i^*)=E_{\infty}(xe\sum_i w_iw_i^*)=0$. In other words, the map $x \to xe$ is injective for $x\in P_{\infty}$. Note that $ee_Pe=\widehat{E}(e_P)e=(\Index E)^{-1}e$.
Then 
\[
\begin{split}
e&=(\Index E)\widehat{E}_{\infty}(e_Pe)\\
&=(\Index E)^2 \widehat{E}_{\infty}(e_Pee_Pe)\\
&=(\Index E)^2\widehat{E}_{\infty}(E_{\infty}(e)e_Pe)\\
&=(\Index E)E_{\infty}(e)e. 
\end{split}
\]
Since $P_{\infty} \ni x \to xe$ injective, $(\Index E)E_{\infty}(e)=1$.\\
(2): Suppose that $E$ is approximately representable with a projection $e\in P_{\infty}\cap P'$ satisfying that $exe=E(x)e$ for any $x$ in $A$. Let $\{u_1,\dots, u_n\}$ be a finite set of $A$ such that 
\begin{equation}
\sum_i u_ieu_i^*=1.
\end{equation}
Define an element $f$ in $B_{\infty}$ by
\[f=\sum_i^n u_i ee_Pu_i^*.\] Since $\{(u_i,u_i^*)\}$ is a quasi-basis for $E$, by \cite[Remark 2.2 (4)]{OKT:Rokhlin} it commutes with $A$. Moreover, $f$ commutes with $e_P$. Therefore, $f\in B_{\infty}\cap B'$. Since
\[\widehat{E}_{\infty}(f)= (\Index E)^{-1}\sum_{i=1}^n u_i eu_i^*=(\Index E)^{-1},\]
it remains to show that the map $B\ni x \to xf$ is injective. Suppose that $xf=0$. We may assume $x$ is of the form $ae_Pb$ for some $a,b$ in $A$. Then \[
\begin{split}
xf&=ae_Pb \sum u_iee_Pu_i^*\\
&=a\sum_i e_Pbu_ie_Peu_i^*\\
&=\sum aE(bu_i)e_Peu_i^*.
\end{split}
\]
By taking $\widehat{E}_{\infty}$, we have $\sum_i aE(bu_i)eu_i^*=0$. Then 
\[\psi(\sum_i aE(bu_i)e_Pu_i^*)=\sum_i aE(bu_i)eu_i^*=0\] where $\psi$ is the map in Proposition \ref{P:approximaterepresentable}.
By the injectivity of $\psi$, 
\[ \begin{split}
0&=\sum_i aE(bu_i)e_Pu_i^*\\
&=\sum_i ae_Pbu_ie_Pu_i^*\\
&=x\sum_i u_ie_Pu_i^*=x.
\end{split}
\] 
Conversely, suppose that $\widehat{E}$ has the Rokhlin property with a Rokhlin projection $f$ in $B_{\infty}\cap B'$. Define an element  $e$ in $A_{\infty}$ by $e=(\Index E)\widehat{E}_{\infty}(fe_P)$. Then we can check that $e$ is a projection in $P_{\infty}\cap P'$ and satisfy the covariance relation for $E$.  Also, it is easy to show that $P\ni x \to xe$ is injective. For details, see \cite[Proposition 3.4]{OKT:Rokhlin}. Now take a quasi-basis $\{(u_i,u_i^*)\}$ for $E$. Then 
\[
\begin{split}
\sum_i u_i eu_i^*&=\sum_i u_i(\Index E)\widehat{E}_{\infty}(fe_P)u_i^*\\
&=(\Index E)\sum \widehat{E}_{\infty}(u_ife_Pu_i^*)\\
&=(\Index E)\widehat{E}_{\infty}(f(\sum_i u_i e_Pu_i^*))\\
&=(\Index E)\widehat{E}_{\infty}(f)=1.
\end{split}
\]
\end{proof}
Finally we are ready to give a reformulation of approximate representability in term of inclusion of $C\sp*$-algebras.  
\begin{prop}\label{P:approxirepresentableviainclusions}
Let $G$ be a finite abelian group and $A$ a unital simple $C\sp*$-algebra. Suppose that $\alpha:G \curvearrowright A$ be an action such that the crossed product $A\rtimes_{\alpha}G$ is simple. Then $\alpha$ is approximately representable if and only if the conditional expectation $E:A \to A^{\alpha}$ defined by $\displaystyle E(a)=\frac{1}{|G|} \sum_{g \in G} \alpha_g(a)$ is approximately representable. 
\end{prop}
\begin{proof}
 Note that $A^{\alpha}$ is stably isomorphic to the cross product $A\rtimes_{\alpha}G$, where the latter is simple.  It follows that $\alpha$ is saturated, thus $E$ is of finite index and $\Index E= |G|$.
Since the kernel of the map $A^{\alpha} \ni x \to xe$ is trivial, it is injective. 
Now if $\alpha$ is approximately representable, then $\widehat{\alpha}$ has the Rokhlin property by Theorem \ref{T:dualityoffiniteabeliangroup}. Thus the conditional expectation from $A\rtimes_{\alpha} G$ to $A$ has the Rokhlin property by Proposition \ref{P:ERokhlin}. But this is a dual inclusion of $A^{\alpha} \subset A$ so that the inclusion $A^{\alpha} \subset A$ is approximately representable by Theorem \ref{T:dualityofinclusion}. The converse follows by the same reverse argument.  
\end{proof}

\section{Dualities of actions and inclusions:The tracial case}\label{S:tracial}

In this section we recall tracial versions of Rokhlin property and approximate representability for  finite group actions and presents tracial versions of such notions for inclusions of $C\sp*$-algebras with a finite index and establish a duality result between two such notions. 
\subsection{Actions on $C\sp*$-algebras}
N.C. Phillips defined a tracial version of the Rokhlin property of a finite group action in \cite{Phillips:tracial}. 

\begin{defn}(Phillips\cite{Phillips:tracial})
Let $G$ be a finite group and $A$ be an infinite dimensional separable unital $C\sp*$-algebra. We say that $\alpha: G \curvearrowright A$ has the tracial Rokhlin property if  for every finite set $F \subset A$, every $\epsilon>0$, any nonzero positive element $x\in A$ there exist $\{e_g\}_{g\in G}$ mutually orthogonal projections such that 
\begin{enumerate}
\item $\| \alpha_g(e_h)-e_{gh} \| \le \epsilon$, \quad $\forall g, h \in g$,
\item $\|  e_ga -ae_g \| \le \epsilon$, \quad $\forall g \in G$, $\forall a \in F$,
\item  Write $e=\sum_{g} e_g$, and $1-e$ is Murray-von Neumann equivalent to a projection in $\overline{xAx}$. 
\end{enumerate}
\end{defn}

Then we can reformulate the tracial Rokhlin property of $\alpha: G \curvearrowright A$ in terms of exact relations using the central sequence algebra. 

\begin{thm}\label{T:tracialRokhlinaction}
Let $G$ be a finite group and $A$ be a separable unital $C\sp*$-algebra. Then $\alpha: G \curvearrowright A$ has the tracial Rokhlin property if  and only if for any nonzero positive element $x \in A_\infty$ there exist a mutually orthogonal  projections  $e_g$'s in $A_{\infty}\cap A'$ such that 
\begin{enumerate}
\item $\alpha_{\infty, g}(e_h)=e_{gh}$, \quad $\forall g,h\in G$ where $\alpha_{\infty}:G \curvearrowright A_{\infty} $ is the induced action,  
\item 
$1-\sum_{g} e_g$ is Murray-von Neumann equivalent to a projection in $\overline{xA_{\infty}x}$.
\end{enumerate}
\end{thm}

We recall that the strict Rokhlin property of $\alpha: G \curvearrowright A$ even for a compact group $G$ can be rephrased as follows. 
\begin{defn}(cf. \cite{BS}, \cite{HW:Rokhlin})
Let $A$ be a separable unital $C\sp*$-algebra and $G$ a second countable, compact group. Let $\sigma : G \curvearrowright􏰏 C(G)$ denote the canonical $G$-shift, that is, $\sigma_g(f) = f(g^{-1} \cdot)$ for all $f \in C(G)$ and $g\in G$. A continuous action $\alpha : G 􏰏\curvearrowright A$ is said to have the Rokhlin property if there exists a unital equivariant $*$-homomorphism
\[(C(G), \sigma)\to (A_{\infty}\cap A' , \alpha_{\infty}).\]
\end{defn}

In the following we demonstrate the same perspective still holds in the case of tracial Rokhlin action of a finite group. 

\begin{thm}\label{T:tracialRokhlinaction2}
Let $G$ be a finite group and $A$ a separable unital infinite dimensional $C\sp*$-algebra. Then $\alpha$ has the tracial Rokhlin property if  and only if for every nonzero positive element $x$ in $A_{\infty}$ there exists an equivariant $*$-homomorphism $\phi$ from $(C(G), \sigma)$ to $(A_{\infty}\cap A', \alpha_{\infty})$ such that $1-\phi(1_{C(G)})$ is Murray-von Neumann equivalent to a projection in a hereditary $C\sp*$-subalgebra generated by $x$ in $A_{\infty}$.  
\end{thm}

\begin{proof}
``$\Longrightarrow$'': Based on Theorem \ref{T:tracialRokhlinaction}, for any nonzero positive $x\in A_{\infty}$ we can take mutually orthogonal projections $e_g$'s in $A_{\infty}\cap A'$ such that $1-\sum e_g \precsim x$. Then we define $\phi(f)=\sum_g f(g)e_g$ for $f\in C(G)$. Then $1-\phi(1_{C(G)})=1-\sum_g e_g \precsim x$ .  Using the condition (1) in Theorem \ref{T:tracialRokhlinaction}, it is easily shown that $\phi$ is equivariant. \\
``$\Longleftarrow$'': Let $x$ be a nonzero positive element in $A_{\infty}$ and suppose that we have an equivariant $*$-homomorphism $\phi:(C(G),\sigma) \to (A_{\infty}\cap A', \alpha_{\infty})$. Let $\chi_g$ be the characteristic function on a singleton $g$. Then $e_g=\phi(\chi_{g})$ is a projection in $A_{\infty}\cap A'$ such that $e_g \perp e_h$ for $g\neq h$ and $1-\sum_g e_g=1-\phi(1_{C(G)})\precsim x$. Moreover,  $\alpha_{\infty,g}(e_h)=\alpha_{\infty, g}(\phi(\chi_h))=\phi(\sigma_g (\chi_h))=\phi(\chi_{gh})=e_{gh}$, so we are done. 
\end{proof}

\begin{lem}\label{L:firstfactorembedding}
Let $G$ be a finite group and both $C$ and $A$ unital $C\sp*$-algebras. Let $\beta:G \curvearrowright C$ and $\alpha:G \curvearrowright A$ be actions of $G$ on $C$ and $A$ respectively. Suppose that for any positive nonzero element $x$ in $A_{\infty}$ there exists an equivarinat $*$-homomorphism $\phi: (C, \beta) \to (A_{\infty}\cap A', \alpha_{\infty})$ such that  $1-\phi(1_C)$ is Murray-von Neumann equivalent to a projection in $\overline{xA_{\infty}x}$. Then $\id_A\otimes 1_C: (A,\alpha) \to (A\otimes C, \alpha\otimes \beta)$ is $G$-tracially sequentially-split. Moreover, the converse is true.
\end{lem}
\begin{proof}
The proof is a straightforward generalization from the nonequivariant case Proposition \ref{P:firstfactorembedding}.
\end{proof}
\begin{cor}\label{C:tracialRokhlinviasequentiallysplitmap}
Let $G$ be a finite group and $\alpha:G \curvearrowright A$ an action of $G$ on a simple unital infinite dimensional $C\sp*$-algebra $A$. Then $\alpha$ has the tracial Rokhlin property if and only if the map $1_{C(G)}\otimes \id_A : (A, \alpha) \to (C(G)\otimes A, \sigma \otimes \alpha)$ is $G$-tracially sequentially split. 
\end{cor}
\begin{proof}
It follows from Theorem \ref{T:tracialRokhlinaction2} and Lemma \ref{L:firstfactorembedding}.
\end{proof}

  We denote by $\phi\rtimes G$ a map from $A\rtimes_{\alpha} G$ to $B\rtimes_{\beta} G$ as a natural extension of an equivariant $*$-homomorphism $\phi:(A,\alpha) \to (B,\beta)$ where $\alpha:G \curvearrowright A$ and $\beta:G \curvearrowright B$. In the following, we denote by $\lambda^{\alpha}:G \to U(A\rtimes_{\alpha}G)$ the implementing unitary representation for the action $\alpha$ so that we write an element of $A\rtimes_{\alpha}G$ as $\sum_{g\in G}a_g\lambda^{\alpha}_g$. The embedding of $A$ into $A\rtimes_{\alpha}G$ is defined by $a \mapsto a\lambda^{\alpha}_e$ or just $a$ without confusion.  

\begin{lem}\cite[Proposition 1.12]{Phillips:tracial}\label{L:projection}
Let $A$ be an infinite dimensional simple unital $C\sp*$-algebra with the property (SP), and $\alpha:G\curvearrowright A$ be an action of a finite group $G$ on $A$ such that $A\rtimes_{\alpha}G$ is also simple.  Let $B \subset A\rtimes_{\alpha}G$ be a nonzero hereditary $C\sp*$-subalgebra. Then there exists a nonzero projection $p\in A$ which is Murray-von Neumann equivalent to a projection in $B$ in $A\rtimes_{\alpha}G$.  
\end{lem}

\begin{defn}\label{D:tracialapproximaterepresentability}
Let $G$ be a finite abelian group and $A$ be an infinite dimensional unital separable simple $C\sp*$-algebra. We say $\alpha:G \curvearrowright A$ is tracially approximately representable if for every positive nonzero element $z$ in $A_{\infty}$, there are a projection $e$ in $A_{\infty}\cap A'$ and a  unitary representation $\omega:G \to eA_{\infty}e$ such that 
\begin{enumerate}
\item $a_g(eae)=\omega_g (eae)\omega_g^*$ in $A_{\infty}$,
\item $\alpha_{\infty, g}(\omega_h)=\omega_h$ for all $g,h\in G$,
\item $1-e$ is Murray-von Neumann equivalent to a projection $\overline{zA_{\infty}z}$.
\end{enumerate}
\end{defn}
\begin{rem}
Of course, in this case also we have a dichotomy that $A$ has the property (SP) or $\alpha$ is approximately representable.
\end{rem}

\begin{prop}\label{P:approximaterepresentabilityintermsofmaps}
Let $G$ be a finite abelian group and $A$  an infinite dimensional unital separable simple $C\sp*$-algebra. Then $\alpha:G \curvearrowright A$ is tracially approximately representable if and only if for every nonzero positive element $z$ in $A_{\infty}$ there are a projection $e\in A_{\infty}\cap A'$ and an equivariant $*$-homomorphism $\psi:(A\rtimes_{\alpha}G, \Ad \lambda^{\alpha}) \to (A_{\infty}, \alpha_{\infty})$ such that 
\begin{enumerate}
\item $\psi(\iota_A(a))=ae$ for all $a\in A$, $\psi(\lambda^{\alpha}_g)=\omega_g$, 
\item $1-\psi(1_A)=1-e$ is Murray-von Neumann equivalent to a  projection in $\overline{zA_{\infty}z}$. 
\end{enumerate}
\end{prop}
\begin{proof}
Suppose that $\alpha:G \curvearrowright A$ is tracially approximately representable. Choose a nonzero positive element $z$ in $A_{\infty}$. Then we have a projection $e\in A_{\infty}\cap A'$ and a unitary representation $\omega:G \to eA_{\infty}e$ satisfying two conditions as in Definition \ref{D:tracialapproximaterepresentability}. We consider the map $\phi:A\ni x \to xe\in A_{\infty}$. Then it is easily checked that $(\phi, \omega)$ gives a covariant pair for $(A, \alpha)$ so that it induces a map $\psi$ from $A\rtimes_{\alpha}G$ to $A_{\infty}$ such that $\psi(\lambda^{\alpha}_g)=\omega_g$ and $\psi(\iota_A(a))=ae$ for all $a$ in $A$. Since $1-e$ is Murray-von Neumann equivalent to a projection in $\overline{zA_{\infty}z}$, so is $1-\phi(1_A)$. The equivariantness of $\psi$ follows from 
\[ \alpha_{\infty, h}\left(\psi(a\lambda_g^{\alpha})\right)=\alpha_{\infty, h}(a\omega_g)=\alpha_h(a)\omega_g=\psi(\Ad(\lambda^{\alpha}_h)(a\lambda^{\alpha}_g)).\]
Conversely, for any nonzero positive element $z$ in $A_{\infty}$ we have  a projection $e\in A_{\infty}\cap A'$ and an equivariant map $\psi:(A\rtimes_{\alpha}G, \Ad \lambda^{\alpha}) \to (A_{\infty}, \alpha_{\infty})$ such that 
\begin{enumerate}
\item $\psi(\iota_A(a))=ae$ for all $a\in A$, $\psi(\lambda^{\alpha}_g)=\omega_g$, 
\item $1-\psi(1_A)=1-e$ is Murray-von Neumann equivalent to a  projection in $\overline{zA_{\infty}z}$. 
\end{enumerate}
We put $\omega_h=\psi(\lambda^{\alpha}_h)$ for $h\in G$. Note that $\omega_h \in U(eA_{\infty}e)$. The equivariantness of $\psi$  implies that $\alpha_g(eae)=\omega_g ae \omega_g^*$ for any $a\in A$.
\end{proof}

\begin{cor}\label{C:traciallyapproxintermsofsequentiallysplit}
Let $G$ be a finite abelian group and $A$ be a unital separable simple $C\sp*$-algebra.  Then $\alpha:G \curvearrowright A$ is tracially approximately representable if and only if $\iota_A: (A, \alpha) \to (A\rtimes_{\alpha}G, \Ad(\lambda^{\alpha}))$ is $G$-tracially sequentially-split. 
\end{cor}

\begin{lem}\label{L:stablesequentialsplit}
Let $G$ be a finite group and $\alpha:G \curvearrowright A$ and $\beta:G \curvearrowright B$  two actions on unital separable simple infinite dimensional $C\sp*$-algebras $A$ and $B$ respectively. Then $\phi:(A, \alpha) \to (B, \beta)$ is $G$-tracially sequentially-split if and only if $\phi\otimes \id_{M_n}: (A\otimes M_n, \alpha \otimes \rho) \to (B\otimes M_n, \beta\otimes \rho)$ is $G$-tracially sequentially-split for $n=|G|$. Here, in fact, $\mc{K}(l^2(G))=M_n$.
\end{lem}
\begin{proof}
Since the strict case follows from \cite [Proposition 3.14]{BS},  we may assume that $A$ has the property (SP).  Note that  $A_{\infty}\otimes M_{n} \cong (A\otimes M_n)_{\infty}$ by the map $[(a_n)_n]\otimes e_{ij} \mapsto [(a_n\otimes e_{ij})_n]$. Suppose that $\phi:(A, \alpha) \to (B, \beta)$ is $G$-tracially sequentially-split.  Consider a nonzero positive element $z$ in $(A\otimes M_n)_{\infty}$. 
Since $(A\otimes M_n)$ is also simple and has the property (SP), there is a projection $p'$ of the form $p\otimes 1$ in $A_{\infty}\otimes M_n$ such that $p'$ is Murray-von Neumann equivalent to a projection in $\overline{z(A\otimes M_n)_{\infty}z}$ (see \cite[Lemma 1.11]{Phillips:tracial}). Then we take a tracial approximate left inverse $\psi:(B,\beta) \to (A_{\infty}, \alpha_{\infty})$ and a projection $e\in A_{\infty}\cap A'$ such that
\begin{enumerate}
\item $\psi(\phi(a))=ae$, 
\item $1-e$ is Murray-von Neumann equivalent to a projection $\overline{p'A_{\infty}p'}$.
\end{enumerate}  

Then $1_{A_{\infty}\otimes M_n} - (e\otimes 1)=(1-e)\otimes 1 \lesssim p' \lesssim z$.  Note that $e \otimes 1$ is in $(A\otimes M_n)_{\infty}\cap (A\otimes M_n)'$ and $1_{A_{\infty}\otimes M_n}- \psi\otimes \id_{M_n}(1)$ is Murray-von Neumann equivalent a projection in $\overline{z(A\otimes M_n)_{\infty}z}$. Also, 
\[ \psi\otimes \id_{M_n}(\phi\otimes \id_{M_n}(a\otimes e_{ij}))=ae\otimes e_{ij}=(a\otimes e_{ij})(e\otimes 1). \] So $\phi\otimes \id_{M_n}$ is $G$-tracially sequentially-split. \\
Conversely, suppose that $\phi \otimes \id_{M_n}$ is $G$-tracially sequentially-split. Take any nonzero positive element $z\in A_{\infty}$, and consider a tracial approximate left inverse $\widetilde{\psi}$ corresponding to $z\otimes e_{11}$. Then we define $\psi:B \to A_{\infty}$ by the restriction of $\widetilde{\psi}$ to $B\otimes 1$. Since $(A_{\infty}\otimes M_n) \cap (A\otimes M_n)' \subset (A_{\infty}\otimes M_n) \cap (1\otimes M_n)'= A_{\infty} \otimes 1$, $(A_{\infty}\otimes M_n) \cap (A\otimes M_n)' =(A_{\infty}\cap A')\otimes 1$. It follows that $\widetilde{\psi}(1)= g\otimes 1 $ where $g\in A_{\infty}\cap A'$. Since $1-\widetilde{\psi}(1)=(1-g)\otimes 1$ is Murray-von Neumann equivalent to a projection in $\overline{zA_{\infty}z}\otimes e_{11}M_ne_{11}$, $1-g$ is Murray-von Neumann equivalent to a projection in $\overline{zA_{\infty}z}$. Also, we see that by viewing $A=A\otimes 1$
\[\psi(\phi(a))=\psi(\phi(a)\otimes1)=\widetilde{\psi}((\phi(a)\otimes1)=(a\otimes 1)(g\otimes1)=ag.\]
Thus, $\psi$ is a $G$-tracial approximate left inverse for $\phi$ corresponding to $z$.  
\end{proof}

We are ready to prove the following duality result for  $G$-tracially sequentially-split maps in the case of $G$ a finite abelian group as one of our main results. 
 
\begin{thm}\label{T:dualityoftraciallysequentiallysplitmap}
Let $G$ be a finite abelian group and $A$ and $B$  infinite dimensional unital separable simple $C\sp*$-algebras where $\alpha$ and $\beta$ acts on respectively.  Further we assume that $\alpha:G \curvearrowright A$ is an action  such that $A\rtimes_{\alpha}G$ is simple, in particular an outer action.  Then  the equivariant $*$-homomorphism $\phi:(A, \alpha) \to (B, \beta)$ is $G$-tracially sequentially-split if and only if $\widehat{\phi}=\phi\rtimes G:(A\rtimes_{\alpha}G, \widehat{\alpha}) \to (B\rtimes_{\beta}G, \widehat{\beta})$ is $\widehat{G}$-tracially sequentially-split. 
\end{thm}
\begin{proof}
Since the strict case follows from \cite [Proposition 3.14]{BS},  we may assume that $A$ has the property (SP). Suppose that an equivariant map $\phi:(A, \alpha) \to (B, \beta)$ is $G$-tracially sequentially-split and consider a nonzero positive element $z$ in $(A\rtimes_{\alpha}G)_{\infty}$. By Lemma \ref{L:projectioninsubalgebra} there is a projection $p$ in $A_{\infty}$ which is Murray-von Neumann equivalent to a projection $r$ in $\overline{z(A\rtimes_{\alpha}G)_{\infty}z}$. Then we can take a tracial approximate left inverse $\psi$ such that $1- \psi(1)$ is Murray-von Neumann equivalent to a projection in $pA_{\infty}p$. Let us denote $e$ by $\psi(1)$. Then 
\[1-e \lesssim p \sim r \in \overline{z(A\rtimes_{\alpha}G)_{\infty}z}.\] Thus $1-e$ is Murray-von Neumann equivalent to a projection in $\overline{z(A\rtimes_{\alpha}G)_{\infty}z}$.\\
From the equivariantness of $\psi$, $e$ is invariant under the action of $\alpha_{\infty}$ so that  
\[ (\psi \rtimes G)(\phi(a)\lambda^{\beta}_g)=\psi(\phi(a))\lambda_{g}^{\alpha_{\infty}}=(ae)\lambda_{g}^{\alpha_{\infty}}=
a\lambda^{\alpha_{\infty}}_g e. \]
Moreover,  via $A_{\infty} \hookrightarrow A_{\infty}\rtimes_{\alpha_{\infty}}G \hookrightarrow (A\rtimes_{\alpha}G)_{\infty}$
\[1-(\psi\rtimes G)(1)=1-\psi(1)=1-e.\]
Conversely, suppose that the equivariant map $\phi \rtimes G:(A \rtimes_{\alpha}G, \widehat{\alpha}) \to (B \rtimes_{\beta}G, \widehat{\beta})$ is $\widehat{G}$-tracially sequentially-split. Then by the above proof we know that $\widehat{\widehat{\phi}}: (A\rtimes_{\alpha}G\rtimes_{\widehat{\alpha}}\widehat{G}, \widehat{\widehat{\alpha}}) \to (B\rtimes_{\beta}G\rtimes_{\widehat{\beta}}\widehat{G}, \widehat{\widehat{\beta}}) $ is $G$-tracially sequentially-spilt. 
Note that by Takai duality \cite{T:duality} there are  equivariant isomorphisms, where $\rho$ is the $G$-action on the algebra of compact operators $\mc{K}(l^2(G))=M_n$ induced by the right-regular representation,   
\[\kappa_A: (A\rtimes_{\alpha}G\rtimes_{\widehat{\alpha}}\widehat{G}, \widehat{\widehat{\alpha}}) \cong (A\otimes M_n, \alpha\otimes \rho)\] 
\[\kappa_B:(B\rtimes_{\beta}G\rtimes_{\widehat{\beta}}\widehat{G}, \widehat{\widehat{\beta}}) \cong (B\otimes M_n, \beta\otimes \rho)\]
such that  the following commutative diagram is commutative:
\begin{equation*}\label{D:Takai}
\xymatrix{ (A\rtimes_{\alpha}G\rtimes_{\widehat{\alpha}}\widehat{G}, \widehat{\widehat{\alpha}})  \ar[d]_{\kappa_A}\ar[r]^{\widehat{\widehat{\phi}}} & (B\rtimes_{\beta}G\rtimes_{\widehat{\beta}}\widehat{G}, \widehat{\widehat{\beta}}) \ar[d]^{\kappa_B}\\
                         (A\otimes M_n, \alpha\otimes \rho)\ar[r]^{\phi\otimes \id_{M_n}} & (B\otimes M_n, \beta\otimes \rho)} 
\end{equation*}
Note that the isomorphism $\kappa_A$ induces the isomorphism denoted by $(\kappa_A)_{\infty}$ between $((A\rtimes_{\alpha}G\rtimes_{\widehat{\alpha}}\widehat{G})_{\infty}, (\widehat{\widehat{\alpha}})_{\infty}) $ and $((A\otimes M_n)_{\infty}, (\alpha\otimes \rho)_{\infty})$.
Now for a positive nonzero element $z$ in $(A\otimes M_n)_{\infty}$ consider $(\kappa_A)_{\infty}^{-1}(z)=\tilde{z}$ in $(A\rtimes_{\alpha}G\rtimes_{\widehat{\alpha}}\widehat{G})_{\infty}$. Since $\widehat{\widehat{\phi}}$ is $G$-tracially sequentially-split, we have a tracial approximate left inverse $\psi$ from $(B\rtimes_{\beta}G\rtimes_{\widehat{\beta}}\widehat{G}, \widehat{\widehat{\beta}})$ to $((A\rtimes_{\alpha}G\rtimes_{\widehat{\alpha}}\widehat{G})_{\infty}, (\widehat{\widehat{\alpha}})_{\infty})$ such that 
\begin{enumerate}
\item $\psi(\widehat{\widehat{\phi}}(x))=x g$ for $x\in A\rtimes_{\alpha}G\rtimes_{\widehat{\alpha}}\widehat{G}$ where $g=\psi(1) \in (A\rtimes_{\alpha}G\rtimes_{\widehat{\alpha}}\widehat{G})_{\infty} \cap (A\rtimes_{\alpha}G\rtimes_{\widehat{\alpha}}\widehat{G})'$,
\item $1-g$ is Murray-von Neumann equivalent to a projection $r$ in $\overline{\tilde{z}(A\rtimes_{\alpha}G\rtimes_{\widehat{\alpha}}\widehat{G})_{\infty} \tilde{z}}$.
\end{enumerate}
Now consider the map $\tilde{\psi}=(\kappa_A)_{\infty} \circ \psi \circ (\kappa_B)^{-1}$. Then it is equivariant since all three maps are. For any $a\in A$, 
\[
\begin{split}
(\tilde{\psi} \circ (\phi\otimes \id_{M_n}))(a\otimes e_{ij})&=((\kappa_A)_{\infty} \circ \psi \circ \kappa_B^{-1} \circ (\phi \otimes \id_{M_n}))(a \otimes e_{eij}) \\
&=((\kappa_A)_{\infty} \circ \psi \circ \widehat{\widehat{\phi}}\circ \kappa_A^{-1})(a\otimes e_{ij})\\
&=(\kappa_A)_{\infty}( \kappa_A^{-1}(a\otimes e_{ij})g)\\
&=(a\otimes e_{ij})((\kappa_A)_{\infty}(g)).
\end{split}
\] 
Moreover, 
$1-(\kappa_A)_{\infty}(g)=(\kappa_A)_{\infty}(1-g)$ is Murray-von Neumann equivalent to a projection $(\kappa_A)_{\infty}(r)$ in $\overline{z(A\otimes M_n)_{\infty}z}$. It follows  that $\phi \otimes \id_{M_n}$ is $G$-tracially sequentially-split. By Lemma \ref{L:stablesequentialsplit} we conclude that $\phi$ is $G$-tracially sequentially-split. 
\end{proof}

Then we provide an alternative proof based on Theorem \ref{T:dualityoftraciallysequentiallysplitmap} and Takai duality for the following result of N.C. Phillips which is a tracial version of Izumi's result. 
  
\begin{thm}[N.C.Phillips]\label{T:dualityofgroupaction}
 Let $A$ be an infinite dimensional simple separable unital C*- algebra, and let $\alpha: G \curvearrowright A$ be an action of a finite abelian group G on A such that $A\rtimes_{\alpha}G$ is also simple. Then
\begin{enumerate}
\item $\alpha$ has the tracial Rokhlin property if and only if $\widehat{\alpha}$􏰐 is tracially approximately representable.
\item $\alpha$ is tracially approximately representable if and only if $\widehat{\alpha}$􏰐 has the tracial Rokhlin property.
\end{enumerate}
\end{thm}
\begin{proof}
(1): Suppose that $\alpha$ has the tracial Rokhlin property. Then by Corollary \ref{C:tracialRokhlinviasequentiallysplitmap}  the map $1_{C(G)}\otimes \id_A: (A, \alpha) \to (C(G)\otimes A, \sigma\otimes \alpha)$ is $G$-tracially sequentially-split. Thus Theorem \ref{T:dualityoftraciallysequentiallysplitmap} implies that $(1_{C(G)}\otimes \id_A)\rtimes G: (A\rtimes_{\alpha}G, \widehat{\alpha}) \to ((C(G)\otimes A)\rtimes_{\sigma \otimes \alpha} G, \widehat{\sigma \otimes \alpha})$ is $\widehat{G}$-tracially sequentially-split. This means that for every nonzero positive element $z$ in $(A\rtimes_{\alpha}G)_{\infty}$ there are a projection $g$ in the sequence algebra of $A\rtimes_{\alpha}G$ and a corresponding equivariant tracial approximate inverse $\psi$ from $((C(G)\otimes A)\rtimes_{\sigma \otimes \alpha} G, \widehat{\sigma \otimes \alpha})$ to $((A\rtimes_{\alpha}G)_{\infty}, (\widehat{\alpha})_{\infty})$.

Note that we have the following commutative diagram of equivariant maps (see \cite[Proposition 4.25]{BS});
\begin{equation}\label{D:diagram1}
\xymatrix{ (A\rtimes_{\alpha}G, \widehat{\alpha}) \ar[rd]_{\iota_{A\rtimes_{\alpha}G}} \ar[rr]^{(1_{C(G)}\otimes \id_A)\rtimes G} && ((C(G)\otimes A)\rtimes_{\sigma\otimes \alpha}G, \widehat{\sigma \otimes \alpha}) &\\
                          & ((A\rtimes_{\alpha}G) \rtimes_{\widehat{\alpha}}\widehat{G}, \Ad(\lambda^{\widehat{\alpha}})) \ar[ur]_{\phi} }
\end{equation}
Here $\phi$ is an isomorphism which comes from Takai-duality. Now consider a map $\widetilde{\psi}=\psi \circ \phi$. Then for any $b\in A\rtimes_{\alpha}G$
\[ \begin {split}
(\widetilde{\psi} \circ \iota_{A\rtimes_{\alpha}G})(b)&= (\psi \circ \phi \circ \iota_{A\rtimes_{\alpha}G})(b)\\
&=(\psi \circ (1_{C(G)}\otimes \id_A)\rtimes G)(b)\\
&=bg
\end{split} \]
Moreover, 
$1- \widetilde{\psi}(1)=1-\psi(1)=1-g$ is Murray-von Neumann equivalent to a projection in $\overline{z(A\rtimes_{\alpha}G)_{\infty}z}$. Therefore, we have shown that $\iota_{A\rtimes_{\alpha}G}:(A\rtimes_{\alpha}G, \widehat{\alpha}) \to ((A\rtimes_{\alpha}G)\rtimes_{\widehat{\alpha}}\widehat{G}, \Ad ({\lambda^{\widehat{\alpha}}})$ is $\widehat{G}$-tracially sequentially-split. Then by Corollary \ref {C:traciallyapproxintermsofsequentiallysplit} $\widehat{\alpha}$ is tracially approximately representable. 

Conversely, let $B=A\rtimes_{\alpha}G$.  If $\widehat{\alpha}$ is tracially approximately representable, then $\iota_{B}:(B, \widehat{\alpha}) \to (B \rtimes_{\widehat{\alpha}}\widehat{G}, \Ad(\lambda^{\widehat{\alpha}}))$ is $\widehat{G}$-tracially sequentially-split. We also employ the 
diagram (\ref{D:diagram1}) with $\kappa=\phi^{-1}$ as follows; 
\begin{equation}\label{D:diagram2}
\xymatrix{ (A\rtimes_{\alpha}G, \widehat{\alpha}) \ar[rd]_{(1_{C(G)}\otimes \id_A)\rtimes G} \ar[rr]^{ \iota_{A\rtimes_{\alpha}G}} &&((A\rtimes_{\alpha}G) \rtimes_{\widehat{\alpha}}\widehat{G}, \Ad(\lambda^{\widehat{\alpha}}))  &\\
                          &((C(G)\otimes A)\rtimes_{\sigma\otimes \alpha}G, \widehat{\sigma \otimes \alpha})  \ar[ur]_{\kappa} }
\end{equation}
Then using the same argument as before we can show that $(1_{C(G)}\otimes \id_A)\rtimes G$ is $\widehat{G}$-tracially sequentially-split.  By Theorem \ref{T:dualityoftraciallysequentiallysplitmap}  the map $1_{C(G)}\otimes \id_A: (A, \alpha) \to (C(G)\otimes A, \sigma \otimes \alpha)$ is $G$-tracially sequentially-split. Thus $\alpha$ has the tracial Rokhlin property by Corollary \ref{C:tracialRokhlinviasequentiallysplitmap}.\\ 
(2): Suppose that $\alpha$ is the tracially approximately representable. Then we have the following diagram 
\begin{equation}\label{D:diagram4}
\xymatrix{ (A, \alpha) \ar[rd]_{\iota_A} \ar@{-->}[rr]^{ \iota} &&(A_{\infty}, \alpha_{\infty})  &\\
                          &(A\rtimes_{\alpha}G, \Ad(\lambda^{\alpha}))  \ar[ur]_{\psi} \ar@{}[u]|{\text{tracially}\circlearrowleft}}
\end{equation}
By the duality, 
 \begin{equation}\label{D:diagram5}
\xymatrix{ (A\rtimes_{\alpha}G, \widehat{\alpha}) \ar[rd]_{\iota_A \rtimes G} \ar@{-->}[rr]^{ \iota} &&((A\rtimes_{\alpha}G)_{\infty}, (\widehat{\alpha})_{\infty})  &\\
                          &((A\rtimes_{\alpha}G)\rtimes_{\Ad(\lambda^{\alpha})}G, \widehat{\Ad \lambda^{\alpha}})  \ar[ur]_{\psi \rtimes G} \ar@{}[u]|{\text{tracially}\circlearrowleft}}
\end{equation}
Then we also have the following diagram by \cite[Proposition 4.26]{BS}; 
 \begin{equation*}\label{D:diagram6}
\xymatrix{ (A\rtimes_{\alpha}G, \widehat{\alpha}) \ar[dd]_{1_{C(\widehat{G})}\otimes \id_{A\rtimes_{\alpha}G}}\ar[rd]_{\iota_A \rtimes G} \ar@{-->}[rr]^{ \iota} &&((A\rtimes_{\alpha}G)_{\infty}, (\widehat{\alpha})_{\infty})  &\\
                          &((A\rtimes_{\alpha}G)\rtimes_{\Ad(\lambda^{\alpha})}G, \widehat{\Ad \lambda^{\alpha}})  \ar[ur]_{\psi \rtimes G} \ar@{}[u]|{\text{tracially}\circlearrowleft} \\
(C(\widehat{G})\otimes(A\rtimes_{\alpha}G), \sigma\otimes \widehat{\alpha})\ar[ru]_{\phi} }
\end{equation*} 
This means that the second factor embedding $1_{C(\widehat{G})}\otimes \id_{A\rtimes_{\alpha}G}:(A\rtimes_{\alpha}G, \widehat{\alpha}) \to (C(\widehat{G}) \otimes (A\rtimes_{\alpha}G), \sigma\otimes \widehat{\alpha})$ is $\widehat{G}$-tracially sequentially-split. By Corollary \ref{C:tracialRokhlinviasequentiallysplitmap} $\widehat{\alpha}$ has the tracial Rokhlin property.\\
Conversely if $\widehat{\alpha}$ has the tracial Rokhlin property, then we have the following diagram;
 \begin{equation}\label{D:diagram7}
\xymatrix{ (A\rtimes_{\alpha}G, \widehat{\alpha}) \ar[rd]_{1_{C(\widehat{G})}\otimes \id_{A\rtimes_{\alpha}G}} \ar@{-->}[rr]^{ \iota} &&((A\rtimes_{\alpha}G)_{\infty}, (\widehat{\alpha})_{\infty})  &\\
                          &(C(\widehat{G})\otimes(A\rtimes_{\alpha}G), \sigma\otimes \widehat{\alpha})  \ar[ur]_{\psi } \ar@{}[u]|{\text{tracially}\circlearrowleft}}
\end{equation} 
Then again 
 \begin{equation*}\label{D:diagram8}
\xymatrix{ (A\rtimes_{\alpha}G, \widehat{\alpha}) \ar[dd]_{\iota_A\rtimes G} \ar[rd]_{1_{C(\widehat{G})}\otimes \id_{A\rtimes_{\alpha}G}} \ar@{-->}[rr]^{ \iota} &&((A\rtimes_{\alpha}G)_{\infty}, (\widehat{\alpha})_{\infty})  &\\
                          &(C(\widehat{G})\otimes(A\rtimes_{\alpha}G), \sigma\otimes \widehat{\alpha})  \ar[ur]_{\psi } \ar@{}[u]|{\text{tracially}\circlearrowleft}\\
((A\rtimes_{\alpha}G)\rtimes_{\Ad(\lambda^{\alpha})}G, \widehat{\Ad \lambda^{\alpha}}), \ar[ru]_{\phi^{-1}} }
\end{equation*}
This means that  the map $\iota_A \rtimes G: (A\rtimes_{\alpha}G, \widehat{\alpha}) \to ((A\rtimes_{\alpha}G)\rtimes_{\Ad(\lambda^{\alpha})}G, \widehat{\Ad \lambda^{\alpha}})$ is $\widehat{G}$-tracially sequentially-split. Then the duality implies that $\iota_A: (A, \alpha) \to (A\rtimes_{\alpha}G, \Ad(\lambda^{\alpha}))$ is $G$-tracially sequentially-split. It follows from Corollary \ref{C:traciallyapproxintermsofsequentiallysplit} that $\alpha$ is tracially approximately representable.
\end{proof}
\subsection{Inclusion of $C\sp*$-algebras}
Now we turn to consider inclusions of unital $C\sp*$-algebras.  
  \begin{defn}[Osaka and Teruya \cite{OT1}]
Let $P\subset A$ be an inclusion of unital $C\sp*$-algebras such that a conditional expectation $E:A\to P$ has a finite index. We say $E$ has the tracial Rokhlin property if for every positive element $z\in P_{\infty}$ there is a Rokhlin projection $e\in A_{\infty}\cap A$ so that 
\begin{enumerate}
\item $(\Index E) E_{\infty}(e)=g$ is a projection,
\item $1-g$ is Murray-von Neumann equivalent to a projection in the hereditary subalgebra of $A_{\infty}$ generated by $z$ in $A_{\infty}$, 
\item $A\ni x \to xe \in A_{\infty}$ is injective. 
\end{enumerate} 
\end{defn}

As we notice, the third condition is automatically satisfied when $A$ is simple. As in the case of action with the tracial Rokhlin property, if $P\subset A$ is an inclusion of $C\sp*$-algebras and a conditional expectation $E:A \to P$ of index-finite type has the tracial Rokhlin property, or shortly $P\subset A$ an inclusion with the tracial Rokhlin property, then either $A$ has property(SP) or $E$ has the Rokhlin property (see \cite[Lemma 4.3]{OT1}). Like in the strict case the following observation was obtained by the second author and T. Teruya in \cite{OT1}. 

\begin{prop}\cite[Proposition 4.6]{OT1}\label{P:EtracialRokhlin}
Let $G$ be a finite group, $\alpha$ an action of $G$ on an infinite dimensional finite simple separable unital $C^*$-algebra $A$, and E the conditional expectation defined by $\displaystyle E(a)=\frac{1}{|G|} \sum_{g} \alpha_g(a)$.  Then $\alpha$ has the tracial Rokhlin property if and only if $E$ has the tracial Rokhlin property.
\end{prop}

 We note that in this case $A^{\alpha}$ is strongly Morita equivalent to $A\rtimes_{\alpha}G$, thus if an approximation property is preserved by the  strong Morita equivalence,  and if the inclusion $A^{\alpha} \subset A$ of  finite index is tracially sequentially-split, then such an  approximation property can be transferred to $A\rtimes_{\alpha}G$ from $A$ when $\alpha$ has the tracial Rokhlin property.   
\begin{defn}\label{D:tracialapproximaterepresentability}
Let $P\subset A$ be an inclusion of unital $C\sp*$-algebras and $E:A \to P$ be a conditional expectation of index-finite type. A conditional expectation $E$ is said to be tracially approximately representable if for every nonzero positive element $z\in A_{\infty}$ there exist a projection $e \in P_{\infty}\cap P'$, a projection $r\in A_{\infty}\cap A'$, and a finite set $\{ u_i\} \subset A$ such that 
\begin{enumerate}
\item $eae=E(a)e$ for all $a\in A$,
\item $\sum_i u_ieu_i^*=r$, and $re=e=er$, 
\item the map $P\ni x \mapsto xe$ is injective,
\item $1-r$ is Murray-von Neumann equivalent to a projection in $\overline{zA_{\infty}z}$ in $A_{\infty}$.
\end{enumerate}  
\end{defn}
\begin{prop}
Let $P\subset A$ be an inclusion of unital $C\sp*$-algebras and $E:A \to P$ be a conditional expectation of index-finite type. Suppose that $E$ is tracially approximately representable. Then $A$ has  the property (SP) or $E$ is approximately representable.
\end{prop}
\begin{proof}
If $A$ does not have Property (SP), neither does $A_{\infty}$. Thus there is a nonzero positive element $z$ in $A_{\infty}$ such that a hereditary subalgebra of $A_{\infty}$ generated by $z$ does not have any nonzero projection.  However, by the assumption there is a projection $r\in A_{\infty}\cap A'$ such that $1-r$ is Murray-von Neumann equivalent to a projection  in $\overline{zA_{\infty}z}$. It follows that $1-r=0$. Thus $\sum_i u_i e u_i^*=1$, so $E$ is approximately representable. 
\end{proof}

We need some preparations to prove a main result of this section.
\begin{lem}\cite[Lemma 3.12]{LeeOsaka}\label{L:inclusiontechnical}
Let $p,q$ be two projections in $P_{\infty}$ and $e\in A_{\infty}\cap A'$ be a projection such that $(\Index E) E_{\infty}(e)$ is a projection in $P_{\infty}\cap P'$. If $pe=ep$ and $q\lesssim pe$ in $A_{\infty}$, then $q \lesssim p$ in $P_{\infty}$
\end{lem}

\begin{lem} \cite[Theorem 3.13]{LeeOsaka}\label{L:errorinsmalleralgebra}
Let $P \subset A$ be inclusion of $C\sp*$-algebras of  index-finite type and $A$ separable. Suppose $E:A \to P$ has the tracial Rokhlin property.  Then for any nonzero positive element $z \in P_{\infty}$,  there exists a projection $e$ in a central sequence algebra of $A$ such that $(\Index E)E_{\infty}(e)=g$ is a projection such that $1-g$ is Murray-von Neumann equivalent to a projection in $\overline{zP_{\infty}z}$ in $P_{\infty}$.  
\end{lem}
The following lemma is crucial as an analogous result of Lemma \ref{L:projection} in the case of inclusion of $C\sp*$-algebras. 
\begin{lem}\label{L:projectioninsubalgebra}
Let $P\subset A$ be an inclusion of unital $C\sp*$-algebras and $E:A \to P$ be a conditional expectation of index-finite type and of finite depth. Suppose $A$ has the (SP)-property. Then for any nonzero projection $p \in A$ there is a projection $q$ in $P$ such that $q \lesssim p$.  Moreover, every non-zero hereditary $C^*$-subalgebra of $A$ has a projection which is Murray-von Neumann equivalent to some projection in $P$.
\end{lem}
\begin{proof}
In fact, the assumption satisfies the outer condition in the sense of Kishimoto; for any non-zero positive element $x$ in $A$ and an arbitrary positive number $\epsilon$ there is an element $y$ in $P$ such that  
\[ \| y^*(x-E(x))y  \| < \epsilon, \quad \| y^*E(x)y  \| \ge \| E(x)\| -\epsilon. \] More precisely, we say $E:A\to P$ is \emph{outer} if for any element $x \in A$ with $E(x)=0$ and any nonzero hereditary $C\sp*$-subalgebra $C$ of $A$, 
\[\inf \{\|cxc\| \mid c\in C^{+}, \|c\|=1\}=0.\] See the proof of Theorem 2.1 in \cite{Osaka:(SP)-property} for more details. 
\end{proof}

Then we first show characterizations of  the tracial Rokhlin property and tracial approximate representability for inclusions of unital $C\sp*$-algebras as we have done in Section 3. 
\begin{prop}\label{P:tracialRokhlinpropertyviamap}
Let $P\subset A$ be an inclusion of unital $C\sp*$-algebras and $E:A \to P$ be a conditional expectation of index-finite type and of finite depth. Suppose further $A$ is simple.  Then $E$ has the tracial Rokhlin property if and only if for every nonzero positive element $z$ in $P_{\infty}$ there are a projection $e \in A_{\infty}\cap A'$ and an injective map $\beta:A \to P_{\infty}$ such that 
\begin{enumerate}
\item $ae=\beta(a)e$ for all $a\in A$,
\item $(\Index E) ee_Pe=e$,
\item$ ye=ze$ implies that $y=z$ for all $y,z \in P_{\infty}$, 
\item $1-\beta(1)$ is Murray-von Neumann equivalent to a projection in $\overline{zP_{\infty}z}$ in $P_{\infty}$. 
\end{enumerate}  
\end{prop}
\begin{proof}
Suppose that $E$ has the tracial Rokhlin property.  By Lemma \ref{L:errorinsmalleralgebra}, for any nonzero positive element $z$ in $P_{\infty}$ we can take a projection $e$ in $A_{\infty}\cap A'$ such that 
\[(\Index E)E_{\infty}(e) =g \, \, \text{is a projection in} \, P_{\infty}\cap P' \]  and $1-g$ is equivalent to a projection in $ \overline{zP_{\infty}z}$.

Then we define $\beta(a)=(\Index E)E_{\infty}(ae)$.  Thus we see immediately that $1-\beta(1)=1-g$ is Murray-von Neumann equivalent to a projection in $ \overline{zP_{\infty}z}$. The other conditions are verified as in the proof of Proposition \ref{P:Rokhlinpropertyviamap}.\\
Conversely, consider a nonzero positive element $z$ in $A_{\infty}$. Since we assume $A$ has Property (SP), so does $A_{\infty}$. Then by Lemma \ref{L:projectioninsubalgebra} we have a projection $q$ in $P_{\infty}$ which is Murray-von Neumann equivalent to a projection in  the hereditary $C\sp*$-subalgebra generated by $z$ of $A_{\infty}$. Now for this $q$ there are a projection $e \in A_{\infty}\cap A'$ and an injective map $\beta:A \to P_{\infty}$ with properties (1)-(3) as above such that $1-\beta(1) \lesssim q$. By the similar arguments in  Proposition \ref{P:Rokhlinpropertyviamap}, we see that $\beta(a)=(\Index E)E_{\infty}(ae)$ and $1-(\Index E)E_{\infty}(e)=1-\beta(1) \lesssim q \lesssim z$ in $A_{\infty}$. So we are done.    
\end{proof}

\begin{prop}\label{P:tracialapproximaterepresentabilityviamap}
Let $P\subset A$ be an inclusion of unital $C\sp*$-algebras and $E:A \to P$ be a conditional expectation of index-finite type. $E$ is tracially approximately representable if and only if for every nonzero positive element $z \in A_{\infty}$ there are an injective $*$-homomorphism from $C^*\langle A, e_P \rangle$ to $A_{\infty}$ and a projection $r \in A_{\infty}\cap A'$ such that  
\begin{enumerate}
\item $\psi(x)=xr$ for any $x\in A$, 
\item $\psi(e_P) \in P_{\infty}\cap P'$ and $\psi(e_P)r=r\psi(e_P)=\psi(e_P)$,
\item $1- r$ is Murray-von Neumann equivalent to a projection in $\overline{zA_{\infty}z}$.
\end{enumerate}
\end{prop}
\begin{proof}
Suppose that $E$ is tracially approximately representable and consider a nonzero positive element $z$ in $A_{\infty}$. Since there exists a projection $e\in P_{\infty}\cap P'$ such that $eae=E(a)e$ and the map $x \mapsto xe$ is injective for $x\in P$,   the universal property of $C^*\langle A, e_P \rangle$ induces an injective $*$-homomorphism $\psi$ from $C^*\langle A, e_P \rangle$ to $A_{\infty}$, in fact the image is generated by $A$ and $e$, such that $\psi(xe_Py)=xey$. If we consider $\{u_i\}$ such that $\sum_i u_ieu_i^*=r$ is a projection in $A_{\infty}\cap A'$ and $re=e=er$, then  for any $a\in A$
\[
ae = are=a (\sum_i u_i eu_i^*)e=(\sum_i u_i e u_i^*)ae=\sum_i u_i E(u_i^*a)e. 
\]  
Therefore
\[ \psi(ae_P)=\psi(\sum_i u_i E(u_i^*a)e_P)\] which implies that $a=\sum_i u_iE(u_i^*a)$. Similarly, we can show that $a=\sum_i E(au_i)u_i^*$. It follows that $\{(u_i, u_i^*)\}$ is a quasi-basis for $E$, and thus $\sum_i u_i e_Pu_i^*=1$. Then $1- \psi(\sum_i u_i e_P u_i^*)=1-r$ is Murray-von Neumann equivalent to a projection in $\overline{zA_{\infty}z}$, so we are done.\\
Conversely, if for every nonzero positive element $z$ in $A_{\infty}$ there are an injective map from $C\langle A, e_P \rangle$ to $A_{\infty}$ and a projection $r$ in $A_{\infty}\cap A'$ satisfying three conditions in the above. We let $\psi(e_P)=e$. Then for a quasi-basis $\{(u_i,u_i^*)\}$ 
\[\psi (\sum_i ue_Pu_i^*)= (\sum u_i r e u_i^*r)=\sum_i u_i rer u_i^* =\sum_i u_i e u_i^*=r.\]  
Therefore $1-r$ is Murray-von Neumann equivalent to a projection in $\overline{zA_{\infty}z}$.
Moreover,  since $e_P a e_P=E(a)e_P$ for any $a\in A$ we have $\psi(e_Pae_P)=\psi(E(a)e_P)$. It follows that $e a  e=E(a)e$.  
\end{proof}
We now use the above technical lemmas to prove our main result; we also derive a consequence of it. 
\begin{thm}\label{T:tracialdualtiyforinclusion}
Let $P\subset A$ be an inclusion of unital $C\sp*$-algebras and $E:A \to P$ be a conditional expectation of index-finite type and of finite depth. If we denote by $B$  the basic construction for $E$, then we have a dual conditional expectation $\widehat{E}:B \to A$. Suppose that $A$ is simple. Then 
\begin{enumerate}
\item $E$ has the tracial Rokhlin property if and only if $\widehat{E}$ is tracially approximately representable.
\item $E$ is tracially approximately representable if and only if $\widehat{E}$ has the tracial Rokhlin property.
\end{enumerate}
\end{thm}
\begin{proof}
(1): Assume that $E$ has the tracial Rokhlin property and $A$ has the property (SP). Let $z$ be a nonzero positive element in $B_{\infty}$. By Lemma \ref{L:projectioninasubalgebra}, there is a projection $p$ in $P_{\infty}$ which is Murray-von Neumann equivalent to a projection in $\overline{zB_{\infty}z}$. Let $\{(v_i,v_i^*)\mid i=1,\dots, n\}$ be a quasi-basis for $E$. Then we take mutually equivalent orthogonal projections $r_1,\dots, r_n$ in $pP_{\infty}p$ since $P$ has the property (SP). For one of such projections, we can take a Rokhlin projection $e \in A_{\infty}\cap A'$ such that 
$1-(\Index E)E_{\infty}(e) \lesssim r_i$ for $i=1,\dots, n$. 

Now let $u_i=\sqrt{\Index E} v_i e_P$ and $g=(\Index E)E_{\infty}(e)$. Then we can easily see that 
\[\sum_{i=1}^n u_i eu_i^*=\sum_i v_i g e_pv_i^*.\]
Thus  for any $x\in A$
\[\sum_i u_i\widehat{E}(u_i^*x)e=\sum_i (\Index E)v_ie_P\widehat{E}(e_Pv_i^*x)e=\left(\sum_i v_i e_P v_i^*\right)xe=xe.\]
Similarly, we can show that $\sum_ie\widehat{E}(xu_i)u_i^*=ex$.

Let $r=\sum_i u_i e u_i^*$. Since $ge_P=e_Pg$ and $g\in P_{\infty}\cap P'$,
\[\begin{split}
r^2&=\sum_i v_i ge_Pv_i^* \sum_j v_j ge_pv_j^*\\
&=\sum_{i,j} v_ige_P(v_i^*v_j)e_Pgv_j^*\\
&=\sum_{i,j}v_igE(v_i^*v_j)e_Pgv_j^*\\
&=\sum_j \sum_i v_iE(v_i^*v_j)e_Pgv_j^*\\
&=\sum_j v_je_Pgv_j^*=r
\end{split}\]
We verify that $r\in B_{\infty}\cap B'$. Let $a\in A$.
\[\begin{split}
ra&=\sum_i v_i ge_Pv_i^*a\\
&=  \sum_i v_i ge_P \left(\sum_k E(v_i^*a v_k)v_k^*\right)\\
&= \sum_i v_i\left(\sum_k E(v_i^*a v_k)ge_Pv_k^*\right)\\
&=\sum_k \sum_i v_i E(v_i^*av_k)ge_Pv_k^*\\
&=\sum_kav_kge_Pv_k^*=ar
\end{split}\]
\[\begin{split}
re_P&=\sum_i v_i ge_Pv_i^*e_P=  \sum_i v_i g E(v_i^*)e_P\\
&= \sum_i v_iE(v_i^*)e_Pg=ge_P
\end{split}\]
\[\begin{split}
e_Pr&=\sum_i e_Pv_i e_Pgv_i^*=  \sum_i E(v_i)e_Pgv_i^*\\
&= \sum_i ge_PE(v_i)v_i^*=ge_P
\end{split}\]
Using the same argument in \cite[Proposition 3.4]{OKT:Rokhlin}, we can show that 
\[(\Index E)ee_Pe=e.\]It follows that for $x,y \in A$
\[e(xe_Py)e=xee_Pey=x(\Index E)^{-1}ey=(\Index E)^{-1}xye=\widehat{E}(xe_Py)e.\]
Now denote by $\{e_{ij}\}_{i,j=1}^n$ the matrix units in $M_n$.  In $B_{\infty}\otimes M_n$
\[\begin{split}
(1-r)\otimes e_{11}&=[\sum_i v_i \otimes e_{1i}][ \sum_k (1-g)e_P\otimes e_{kk}][ \sum_j v_j^*\otimes e_{j1}]\\
&\sim [ \sum_k (1-g)e_P\otimes e_{kk}][\sum_j v_j^* \otimes e_{j1}][\sum_i v_i \otimes e_{1i}][ \sum_k (1-g)e_P\otimes e_{kk}]\\
&\lesssim \sum_k (1-g)e_P \otimes e_{kk}\\
&\lesssim \sum_k r_ke_P\otimes e_{kk}\\
&\sim (\sum_k r_ke_p)\otimes e_{11}\\
&\lesssim pe_Pp\otimes e_{11}\\
&\lesssim p\otimes e_{11}.
\end{split} \]
Hence we conclude that $1-r \lesssim p \lesssim z$ in $B_{\infty}$.

Conversely, suppose that $\widehat{E}$ is tracially approximately representable. Take a positive nonzero element $z$ in $A_{\infty} (\subset B_{\infty})$. By Lemma \ref{L:projectioninsubalgebra} there is  a projection $p$ in $P_{\infty}$ which is Murray-von Neumann equivalent to a projection $\overline{zA_{\infty}z}$. For $pe_Pp \in B_{\infty}$, we have a projection $e\in A_{\infty}\cap A'$, a projection $r \in B_{\infty}\cap B'$, and a finite set $\{u_i\} \subset B$ such that 
\begin{eqnarray}\label{condition1}\label{condition2}
eze=\widehat{E}(z)e \quad \forall z \in B,\\
\sum_i u_ieu_i^*=r, \, re=e=er,
\end{eqnarray}
\vspace{-5mm}
\begin{equation}\label{condition3}
1-r  \,\text{is Murray-von Neumann equivalent to a projection in $\overline{pe_ppB_{\infty}pe_Pp}$}.
\end{equation}

From (\ref{condition1}) we have $(\Index E)ee_Pe=e$. Let $w_i=(\Index E)\widehat{E}(e_Pu_i)\in A$.
Then 
\[\begin{split}
\left(\sum_i w_i e w_i^*\right)&=\sum_i (\Index E)^2 ee_P u_ieu_i^*e_Pe\\ 
&=(\Index E)^2 ee_P\left(\sum_i u_i eu_i^*\right)e_Pe\\
&=(\Index E)^2ee_Pre_Pe\\
&=(\Index E)^2ere_Pe\\
&=(\Index E)^2ee_Pe\\
&=(\Index E)e. 
\end{split}
\]
 It follows that $\sum_i ww_i^*=\Index E$ since $A\ni x \to xe$ is injective. Thus $E_{\infty}(\sum_i w_iew_i^*)=(\Index E)E_{\infty}(e)$ which will be denoted by $g$. Using the argument in the proof of \ref{T:dualityofinclusion} we can show that $ge_P=re_P$.

By (\ref{condition3}), there is a partial isometry $v$ in $B_{\infty}$ such that $vv^*=1-r$ and $vv^*=q \in \overline{pe_ppB_{\infty}pe_Pp}$. Set $w=ve_p$. Then $w^*w=(1-r)e_P$ and $ww^*=q_0 \leq q $. Note that $q_0 \in \overline{pP_{\infty}e_Pp}$. Now let  $u=q_0w(1-r)e_P$. It is easily checked that $u^*u=(1-g)e_P$ and $uu^*=q_0$. Note that $u=pce_P$ for some $c\in P_{\infty}$. 
Then consider $\widehat{u}=(\Index E)\widehat{E}_{\infty}(u) \in A_{\infty}$.
\[ \widehat{u}^*\widehat{u}e_P= c^*p pc e_P
=e_Pc^*pce_P=u^*u=(1-g)e_P.\] This implies that $\widehat{u}$ is a partial isometry in $A_{\infty}$ such that $\widehat{u}^*\widehat{u}=1-g$ and $\widehat{u}\widehat{u}^* \leq p$. Since $p$ is Murray-von Neumann equivalent to a projection to some projection in $\overline{xA_{\infty}x}$, so is $1-g$.

(2): Suppose $E$ is tracially approximately representable. Take a nonzero positive element $z$ in  $B_{\infty}$. By Lemma \ref{L:projectioninsubalgebra}, we have a projection $p$ in $A_{\infty}$ such that $p$ is Murray-von Neumann equivalent to a projection in $\overline{zB_{\infty}z}$. For this $p$  there is a projection $e$ in $P_{\infty} \cap P'$ such that 
\begin{equation}\label{E:approxcondition(1)}
exe=E(x)e
\end{equation}
 for any $x\in A$, a projection $g$ in $A_{\infty}\cap A'$,  and a finite set $\{u_i\}$ in $A$ such that 
\begin{equation}\label{E:approxcondition(2)}
\sum_i u_i eu_i^*=g, \sum_i u_iE(u_i^*x)e=xe, \sum_i eE(xu_i)u_i^*=ex \quad \forall x \in A
\end{equation} 
(In fact, $ge=e$ implies that $gxe=(\sum_i u_i eu_i^*)xe=\sum_i u_i E(u_i^*x)e= xe$ for $x \in A$. Similarly, $eg=e$ implies that $\sum_i e E(xu_i)u_i^*=ex$.)
\begin{equation}\label{E:approxcondition(3)}
\text{$1-g$ is Murry-von Neumann equivalent to a projection in $\overline{pA_{\infty}p}$}
\end{equation}
Define $f$ in $B_{\infty}$ by 
\[ f=\sum_i u_ie e_P u_i^*.\] Then $f$ is a projection since
\[ \begin{split}
f^2&=\sum_i  u_i ee_Pu_i^* \sum_j u_jee_Pu_j^*\\
&=\sum_{i,j} u_ie_peu_i^*u_jee_Pu_j^*\\
&=\sum_j \left(\sum_i u_i E(u_i^* u_j)e\right)e_Pu_j^* \quad \text{by (\ref{E:approxcondition(1)})}\\
&=\sum_j u_je  e_Pu_j^* =f \quad \text{by (\ref{E:approxcondition(2)})}.
\end{split}\]

 Note that for any  $a\in A$
\[ \begin{split}
fa&=\sum_i  u_i ee_Pu_i^*a\\
&=\sum_i u_ie_peu_i^*a  \quad (e_Pee_P=E_{\infty}(e)E_P=ee_P)\\
&=\sum_i u_ie_P \left(\sum_j eE(u_i^*a u_j)u_j^*\right) \quad \text{by (\ref{E:approxcondition(2)})}\\
&=\sum_j \left(\sum_i u_iE(u_i^*au_j)e \right) e_Pu_j^* \quad \text{by (\ref{E:approxcondition(2)})}\\
&=\sum_j au_j ee_Pu_j^*=af.
\end{split}\]
\vspace{-5mm}
Moreover, 
\[ \begin{split}
fe_P&=\left(\sum_i  u_i ee_Pu_i^*\right)e_P\\
&=\sum_i u_ieE(u_i^*)e_P\\
&=\sum_i \left(u_i E(u_i^*)e\right)e_p \\
&=ee_P \quad \text{by (\ref{E:approxcondition(2)})}\\
\end{split}
\]
and 
\[\begin{split}
e_Pf&=e_P \left(\sum_i u_i e_Pe u_i^* \right)\\
&=\sum_i e_Pu_ie_P e u_i^*\\
&=\sum_i E(u_i)e_Peu_i^*\\
&=\sum_i e_Pe E(u_i)u_i^*\\
&=e_Pe \quad \text{by (\ref{E:approxcondition(2)})}
\end{split}\]
Thus we showed that $f \in B_{\infty}\cap B'$. Finally \[\widehat{E}_{\infty}(f)=(\Index E)^{-1}\left(\sum_i u_i eu_i^*\right).\]
Therefore $(\Index E)\widehat{E}_{\infty}(f)=g$ is a projection such that 
$1-g$ is Murry-von Neumann equivalent to a projection in $\overline{zB_{\infty}z}$ since $1-g \lesssim p\lesssim z$.
 
Conversely, suppose that $\widehat{E}$ has the tracial Rokhlin property. For any positive nonzero element $z\in A_{\infty}$ we consider a hereditary subalgebra $\overline{zA_{\infty}z}$. Since $A$ has the property (SP), so is $A_{\infty}$. Thus there is a projection $p$ in $\overline{zA_{\infty}z}$ which is also contained in $\overline{zB_{\infty}z}$.   Then we consider  mutually orthogonal nonzero projections $r_1, r_2$ in $pA_{\infty}p$ such that $r_1 \lesssim r_2$. Then we can think of a Rokhlin projection $f'$ for $r_2$. For $f'r_1$ we can take a Rokhlin projection $f\in B_{\infty}\cap B'$ such that $(\Index E)\widehat{E}_{\infty}(f)$ is a projection such that $1-(\Index E)\widehat{E}_{\infty}(f)$ is Cuntz subequivalent to $f'r_1$. Then define an element $e=(\Index E)\widehat{E}_{\infty}(fe_P)$ in $A_{\infty}$.Using the fact that $fe_P=e_pf=(\Index E)\widehat{E}_{\infty}(fe_P)e_P$, we can show that 
$e$ is a projection. By the same argument in \cite[Proposition 3.4]{OKT:Rokhlin}, we can show that $e$ is in $P_{\infty}\cap P'$ and $eae=E(a)e$ for any $a\in A$.  Then take a quasi-basis $\{(u_i,u_i^*)\}$ for $E$. Then \[
\begin{split}
\sum_i u_i e u_i^*&=(\Index E)\sum_i u_i \widehat{E}_{\infty}(fe_P)u_i^*\\
&=(\Index E)\widehat{E}_{\infty}(\sum_i u_ife_Pu_i^*)\\
&=(\Index E)\widehat{E}_{\infty}(f) 
\end{split}
\]   Thus $g=\sum_i u_i e u_i^*$ is a projection in $A_{\infty}\cap A'$. By applying Lemma \ref{L:inclusiontechnical} to $f$,$1-g$, and $f'r_1$, we have $1-g$ is Cuntz subequivalent to $r_2$ in $A_{\infty}$. So $1-g \lesssim r_2 \lesssim p \lesssim z$ in $A_{\infty}$ .
\end{proof}
\begin{cor}
Let $G$ be a finite abelian group, $\alpha$ an outer action of $G$ on an infinite dimensional simple separable unital $C\sp*$-algebra $A$ such that $A\rtimes_{\alpha}G$ is simple, and $E$ as in Proposition  \ref {P:ERokhlin}. Then $\alpha$ is tracially approximately representable if and only if $E$ is tracially approximate representable.
\end{cor}
\begin{proof}
The proof goes exactly same as the proof of Proposition \ref{P:approxirepresentableviainclusions} only replacing Theorem \ref{T:dualityoffiniteabeliangroup}, Proposition \ref {P:ERokhlin}, Theorem \ref{T:dualityofinclusion} with Theorem \ref{T:dualityofgroupaction}, Proposition \ref{P:EtracialRokhlin}, Theorem \ref{T:tracialdualtiyforinclusion} respectively. 
\end{proof}
\section{Acknowledgement}
This research was carried out during the first author's stay at KIAS and his visit to Ritsumeikan University. He would like to appreciate both institutions for excellent supports. The second author also would like to appreciate  KfAS for the worm hospitality when he visited it.  


\end{document}